\documentclass[10pt,twoside]{article}

\usepackage{graphicx,latexsym,euscript,makeidx,color,bm}
\usepackage{amsmath,amsfonts,amssymb,amsthm,thmtools,mathrsfs,enumerate}
\usepackage[colorlinks,linkcolor=blue,anchorcolor=green,citecolor=red]{hyperref}

 \def\sD{\mathscr{D}}    
\def\dbE{\mathbb{E}}     
\def\dbF{\mathbb{F}}   \def\cF{{\cal F}}

 \def\sK{\mathscr{K}}

\def\dbP{\mathbb{P}}     
     
\def\dbR{\mathbb{R}}     
\def\dbS{\mathbb{S}}

\def\ss{\smallskip}             \def\hb{\hbox}
\def\ms{\medskip}              
        \def\lan{\langle}    \def\as{\hbox{\rm a.s.}}
\def\ds{\displaystyle}   \def\ran{\rangle}    \def\tr{\hbox{\rm tr$\,$}}
         
\def\no{\noindent}          
\def\ns{\noalign{\ss}}

\def\rf{\eqref}            
\def\deq{\triangleq}     \def\({\Big (}       
\def\les{\leqslant}      \def\){\Big )}       
\def\ges{\geqslant}      \def\[{\Big[}        

\def\da{\mathop{\downarrow}}
          \def\]{\Big]}        
      \def\q{\quad}        
\def\h{\widehat}         \def\qq{\qquad}      \def\1n{\negthinspace}
\def\cd{\cdot}           \def\2n{\1n\1n}      \def\3n{\1n\1n\1n}

\def\a{\alpha}           \def\g{\gamma}   \def\Om{\Omega}  \def\om{\omega}
\def\b{\beta}         \def\D{\Delta}   \def\d{\delta}   \def\F{\Phi}     
\def\z{\zeta}           \def\th{\theta}  \def\Si{\Sigma}  \def\si{\sigma}
\def\e{\varepsilon}     \def\l{\lambda}  \def\m{\mu}      \def\n{\nu}
    \def\t{\tau}     \def\f{\varphi}  \def\i{\infty}   
\def\x{\xi}           \def\k{\kappa}

\def\ba{\begin{array}}                \def\ea{\end{array}}
\def\bel{\begin{equation}\label}      \def\ee{\end{equation}}
\def\be{\begin{equation}}
\newtheorem{theorem}{Theorem}[section]
\newtheorem{definition}[theorem]{Definition}
\newtheorem{proposition}[theorem]{Proposition}
\newtheorem{corollary}[theorem]{Corollary}
\newtheorem{lemma}[theorem]{Lemma}
\newtheorem{remark}[theorem]{Remark}

\newtheorem{taggedassumptionx}{}
\newenvironment{taggedassumption}[1]
 {\renewcommand\thetaggedassumptionx{#1}\taggedassumptionx}
 {\endtaggedassumptionx}
\makeatletter
   
   \@addtoreset{equation}{section}
\makeatother
  \allowdisplaybreaks[4]

\makeatletter
\evensidemargin \oddsidemargin
\makeatother

\begin{document}
\thispagestyle{empty}
\setcounter{page}{1}


\begin{center}
{\large\bf A FINITE HORIZON OPTIMAL STOCHASTIC IMPULSE CONTROL PROBLEM WITH A DECISION LAG}

\vskip.20in

Chang Li\ and\ Jiongmin Yong \\[2mm]
{\footnotesize
Department of Mathematics\\
University of Central Florida, Orlando, FL, USA\\[5pt]
Corresponding author email: changli@knights.ucf.edu

}
\end{center}

{\footnotesize
\noindent
{\bf Abstract.} This paper studies an optimal stochastic impulse control problem in a finite time horizon with a decision lag, by which we mean that after an impulse is made, a fixed number units of time has to be elapsed before the next impulse is allowed to be made. The continuity of the value function is proved. A suitable version of dynamic programming principle is established, which takes into account the dependence of state process on the elapsed time. The corresponding Hamilton-Jacobi-Bellman (HJB) equation is derived, which exhibits some special feature of the problem. The value function of this optimal impulse control problem is characterized as the unique viscosity solution to the corresponding HJB equation. An optimal impulse control is constructed provided the value function is given. Moreover, a limiting case with the waiting time approaching $0$ is discussed. \\[3pt]

}

\vskip.2in

\section{Introduction}\label{Sec:Introduction}

\noindent Let $(\Om,\cF,\dbF,\dbP)$ be a complete filtered probability space on which a $d$-dimensional standard Brownian motion $W(\cd)$  is defined, with $\dbF$ being its natural filtration augmented by all the $\dbP$-null sets. Consider the following stochastic differential equation (SDE, for short):
\bel{state1}
X(s)=x+\int_t^sb(\t,X(\t))d\t+\int_t^s\si(\t,X(\t))dW(\t)+\xi(s),\q s\in[t,T],\ee
where $b:[0,T]\times\dbR^n\to\dbR^n$ and $\si:[0,T]\times\dbR^n\to\dbR^{n\times d}$ are some suitable deterministic maps, $X(\cd)$ is the {\it state process} with $t\in[0,T)$ being the {\it initial time} and $x\in\dbR^n$ being the {\it initial state}, and $\xi(\cd)$ is called an {\it impulse control} of the following form:
\bel{xi}\xi(s)=\sum_{i\ges1}\x_i\chi_{[\t_i,T]}(s),\q s\in[t,T].\ee
Here, $\{\t_i\}_{i\ges1}$ is an increasing sequence of $\dbF$-stopping times valued in
$[t,T]$, and each $\xi_i$ is an $\cF_{\t_i}$-measurable square integrable random variable taking values in $K$, with $K\subseteq\dbR^n$ being a closed convex cone. Unlike the classical impulse control problems (\cite{Bensoussan-Lions 1984}), for any two consecutive impulses, a gap is required:
\bel{>t}\t_{i+1}-\t_i\ges\d,\qq i\ges1,\qq\as,\ee
for some fixed constant $\d>0$ which is called a {\it decision lag}. In another word, after an impulse is made, another immediate impulse is not allowed. In reality, this makes a perfect sense and one can easily cook up examples for this. For example, contributions to the retirement account (biweekly, monthly, or skipping), adjustment of the portfolio made by a fund manager (monthly, quarterly, or no changes), trading assets in some security market\footnote{In current stock market of China, there is a so-called ``$T+1$'' rule, meaning that buying a stock today, one is not allowed to sell it until tomorrow.}, to mention a few (\cite{ Aouchiche 2011,  Bayraktar-Egami 2007, Bar-Ilan 1996, Alvarez 2002, Stettner 2011}).

\ms

Due to the existence of the decision lag, for any initial time $t\in[0,T)$, whether an impulse at $t$ or shortly after is allowed depends on when was the last impulse made before time $t$. To more precisely describe this, we introduce a variable $r\in[0,T)$, called an {\it elapsed time}, which is defined by the following: Suppose $\t_0$ is the last moment before $t$ at which an impulse was made. Then $r=t-\t_0$ is the elapsed time at $t$. We make a convention that if no impulse has ever been made on $[0,t)$, then set the elapsed time $r=\d\vee t$, so that an immediate impulse at $t$ is allowed. From this, we see that it makes more sense to take $(t,r,x)$ as the {\it initial triple}, with $(t,x)\in[0,T)\times\dbR^n$ being the usual {\it initial pair} and with $r\in[0,T)$ being the {\it initial elapsed time}. We let $\sD$ be the set of all initial triples $(t,r,x)$. Thus,
\bel{sD}\sD=[0,T]\times[0,T]\times\dbR^n.\ee
Now, for any $(t,r,x)\in\sD$, we let $\sK_r[t,T]$ be the set of all impulse controls of the form \rf{xi} with \rf{>t} being true and with the initial elapsed time $r$. Then for any $\xi(\cd)\in\sK_r[t,T]$ of form \rf{xi}, we claim that
\bel{t_1}\t_1\ges(t-r+\d)\vee t.\qq\as\ee
In fact, if the last impulse before $t$ was made at $\t_0$, then
$$\t_1\ges(\t_0+\d)\vee t=(t-r+\d)\vee t,\qq\as;$$
and in the case that no impulse has ever been made in $[0,t)$, by our convention, one has
$$\t_1\ges t=(t\land\d)\vee t=(t-\d\vee t+\d)\vee t=(t-r+\d)\vee t,\qq\as$$
Thus, \rf{t_1} holds. Clearly, the role played by $r$ is in the determination of $\t_1$. Moreover, we see that
\bel{K}\ba{ll}
\ns\ds\sK_{\hat r}[t,T]\subseteq\sK_r[t,T]\subseteq\sK_\d[t,T]=\sK_{r'}[t,T],\\
\ns\ds\qq\qq\qq\qq\qq\qq\qq\forall\,0\les\hat r\les r\les\d\les r'\les T.\ea\ee

Under proper conditions, for any $(t,r,x)\in\sD$ and $\xi(\cd)\in\sK_r[t,T]$, state equation \rf{state1} admits a unique solution $X(\cd)\equiv X(\cd\,;t,r,x,\xi(\cd))$. To measure the performance of the impulse control, we introduce the following cost functional:
\bel{J1}J(t,r,x;\xi(\cd))=\dbE\Big\{\int_t^Tg(s,X(s))ds+ h(X(T))+\sum_{i\ges1}\ell(\t_i,\xi_i)\Big\},\ee
for some suitable deterministic maps $g(\cd\,,\cd)$, $h(\cd)$ and $\ell(\cd\,,\cd)$. The terms on the right-hand side are the {\it running cost}, the {\it terminal cost}, and the {\it impulse cost}, respectively. Our optimal control problem can be stated as follows.

\ms

\bf Problem (IC). \rm For any $(t,r,x)\in\sD$, find a $\bar\xi(\cd)\in\sK_r[t,T]$ such that
\bel{bar xi}J(t,r,x;\bar\xi(\cd))=\inf_{\xi(\cd)\in\sK_r[t,T]}J(t,r,x;\xi(\cd))\equiv V(t,r,x).\ee

\ms

Any $\bar\xi(\cd)\in\sK_r[t,T]$ satisfying \rf{bar xi} is called an {\it optimal impulse control}, and $\bar X(\cd)\equiv X(\cd\,;t,r,x,\bar\xi(\cd))$ is called the corresponding {\it optimal state process}. We call $V(\cd\,,\cd\,,\cd)$ the {\it value function} of Problem (IC).

\ms

Classical optimal impulse control theory can be traced back to the work of Bensoussan--Lions in the early 1970s (\cite{Bensoussan-Lions 1973, Bensoussan-Lions 1984}). Many follow-up literature appeared since then, see \cite{Menaldi 1987, Lenhart 1989, Tang-Yong 1993, Yong 1994, Borkar-Ghosh-Sahay 1999, Menaldi-Robin 2017}, for examples. It is well-known that for a classical impulse control problem, if the state equation is a stochastic differential equation with deterministic coefficients and the cost functional also only involves deterministic functions, then under some mild conditions, the value function of the problem is the unique viscosity solution to a Hamilton-Jacobi-Bellman equation of a quasi-variational inequality form. Once the value function is determined, an optimal impulse control can be constructed, which solves the optimal impulse control problem.

\ms

Optimal impulse control problems with one (fixed) {\it execution lag} were firstly studied by Robin in the middle of 1970s (\cite{Robin 1976}). Unlike problems with decision lag, in a problem with an execution lag, one decides, at some $\t_i$, an impulse $\xi_i$ to be made, which will be realized at a later time $\t_i+\D$ for some fixed lag $\D>0$. Due to the fact that the decision lag $\d\ges0$ might be smaller than the execution lag $\D$, there could be some pending ``orders'', the impulses ordered during $(\t_i,\t_i+\D)$. For details, see \cite{Bruder-Pham 2009}, in which the execution lag is an integer multiple of the decision lag, i.e., $\D=m\d$. To get more feeling, as well as for the purpose of comparison with the results of the current paper, let us consider the situation studied in \cite{Bruder-Pham 2009} of maximum pending order $m=1$ and minimizing the cost instead of maximizing the payoff. Let $v^0(t,x)$ be the optimal value of the cost functional corresponding the initial pair $(t,x)$ with no pending order and $v^1(t,x,(\t,\xi))$ be the optimal value of the cost functional corresponding to the initial pair $(t,x)$ with one pending order $(\t,\x)$ (the impulse ordered at $\t$ with size $\x$ which will be exercised at $\t+\d$), then the corresponding HJB equation system is as follows:
\bel{v^0}\left\{\2n\ba{ll}
\ds \min\Big\{v^0_t(t,x)+H(t,x,v^0_x(t,x),v^0_{xx}(t,x)),\\
\ns\ds \qq\inf_{\x\in K}\big[v^1\big(t,x,(t,\x)\big)\big]\1n-\1n v^0(t,x)\Big\}\1n=\1n0,\q(t,x)\1n\in\1n [0,T\1n-\1n\d]\1n\times\1n\dbR^n,\\
\ns\ds v^0_t(t,x)\1n+\1n H(t,x,v^0_x(t,x),v^0_{xx}(t,x))\1n=\1n0,\q(t,x)\1n\in\1n(T\1n-\1n\d,T)
\1n\times\1n\dbR^n,\\ [1mm]
\ns\ds v^0(T,x)=h(x), \qq\qq~x\in\dbR^n.
\ea\right.\ee
\bel{v^1}\left\{\ba{ll}
\ds v^1_t\big(t,x,(\t,\x)\big)+H\big(t,x,v^1_x(t,x,(\t,\x)),v^1_{xx}(t,x,(\t,\x))
\big)=0,\\ [1mm]
\qq\qq\qq(t,x,(\t,\x))\in [\t, \t+\d)  \times\dbR^n \times ([0,T-\d] \times K),\\
\ns\ds v^1\big((\t+\d)-,x,(\t,\x)\big)=c(x,\x)+v_0(
\t+\d,x+\x),\\ [1mm]
\qq\qq\qq\qq\qq (x,(\t,\x)) \in\dbR^n \times ([0,T-\d]\times K).\ea\right.\ee
Note that the above two equations are coupled in the following way: $v^1(\cd\,,
\cd\,,\cd)$ appears in the obstacle of the equation for $v^0(\cd\,,\cd)$; $v^0(\cd\,,\cd)$ appears in the terminal condition of the equation for $v^1(\cd\,,\cd\,,\cd)$. We will see some common feature and some significant difference between the above and our HJB equation later.

\ms

Let us now briefly recall the main relevant results in several other papers.
Optimal impulse control problem in an infinite horizon with an execution lag was investigated in \cite{Bayraktar-Egami 2007,Oksandal-Sulem 2008}, where no HJB equations were derived. A switching problem with decision lag and execution lag for discrete-time systems was studied in \cite{Aouchiche 2011}, in which, only some numerical algorithms were presented. In \cite{Stettner 2011}, an asymptotic optimization problem of terminal wealth with decision lag or execution lag under HARA utility was studied. Some kind of Bellman dynamic programming equations corresponding to several situations were presented. However, still no HJB equations were derived. An optimal switching problem with a decision lag for ODEs was studied in \cite{Granato-Zidani 2014}. A reachable set was characterized by the level set of the value function which is the unique viscosity solution to a first order HJB equation. In \cite{Hdhiri-Karouf 2018}, an optimal impulse control problem is considered for a general stochastic process (without concrete SDE state equation) with execution lag. Snell envelope and reflected BSDEs were used to obtain the optimal impulse controls. In \cite{Palczewski-Stettner 2010}, an optimal impulse problem in a finite horizon with arbitrary number of pending orders for Feller process were investigated without corresponding HJB equation derived.

\ms

In this paper, we consider the optimal impulse control problem with a decision lag (without execution lag). It should be pointed out that, unlike the above-cited works, we have paid a special attention on the {\it elapsed time} since the last impulse was made. The introduction of the elapsed time $r$ helps us to fully understand the problem. Because of that, our value function is of form $V(t,r,x)$ and therefore, $V_r(t,r,x)$ will naturally appear, which makes our HJB equation significantly different from those in the literature, say, of form \rf{v^0}--\rf{v^1}. We will show directly the continuity of the value function $V(t,r,x)$ in all its arguments, by using some ideas from \cite{Yong 1994}, unlike some indirect and complicated arguments used in \cite{Bruder-Pham 2009}. Then we establish a suitable version of dynamic programming principle using an argument inspired by \cite{Fleming-Souganidis 1989}, which leads to the corresponding HJB equations. We further show that the value function is the unique viscosity solution to the HJB equation, by a technique adopted from \cite{Yong-Zhou 1999}. Moreover, an optimal impulse control is constructed from the given value function. Finally, a limiting case with the decision lag approaching $0$ is discussed, which exactly recovers the classical impulse control problems.

\ms

The remaining part of the paper is organized as follows. Section 2 introduces the value function associated with the control problem and its properties. Section 3 provides a suitable version of dynamic programming principle and derived the corresponding HJB equations. In Section 4, the value function is proved to be the unique viscosity solution of HJB equations in some given function space (with some technical details put in the appendix) and an optimal impulse control was constructed with verification theorem. Finally, Section 5 concludes the paper.

\section{The Value Function and Its Properties}\label{Sec:Value Function}

Recall that $(\Om,\cF,\dbP)$ is a complete probability space on which a standard $d$-dimensional Brownian motion $W(\cd)=\{W(t);0\les t<\i\}$ is defined, with $\dbF=\{\cF_t\}_{t\ges0}$ being its natural filtration augmented by all the $\dbP$-null sets in $\cF$. Let $T>0$ be given and let $K$ be a closed convex cone in $\dbR^n$. For the coefficients of the state equation \rf{state1}, we introduce the following assumption.

\begin{taggedassumption}{(H1)}\label{H1} \rm Let $b:[0,T]\times\dbR^n\to\dbR^n$, $\si:[0,T]\times\dbR^n\to\dbR^{n\times d}$ be continuous and there exists a constant $L>0$ such that, for all $x, \hat{x} \in \dbR^n, t \in [0,T]$,
\bel{b-si}\ba{ll}
\ns\ds|b(t,x)-b(t,\hat x)|+|\si(t,x)-\si(t,\hat x)|\les L|x-\hat x|,\\
\ns\ds|b(t,x)|+|\si(t,x)|\les L.\ea\ee

\end{taggedassumption}

For the functions involved in the cost functional, we introduce the following assumption.

\ms

\begin{taggedassumption}{(H2)}\label{H2} \rm Let
$g:[0,T]\times\dbR^n\to\dbR$, $h:\dbR^n\to\dbR$ and $\ell:[0,T]\times K\to\dbR^+$ be continuous and there are constants $\ell_0,\a>0$ such that, for all $x,\hat x\in\dbR^n$, $0\les t<\hat t\les T$, and $\xi,\hat\xi\in K$,
\bel{gh}\ba{ll}
\ns\ds|g(t,x)-g(t,\hat x)|+|h(x)-h(\hat x)|\les L|x-\hat x|,\\
\ns\ds|g(t,x)|+|h(x)|\les L,\ea\ee
and
\bel{ell}\ba{ll}
\ns\ds\ell(t,\xi+\hat\xi)<\ell(t,\xi)+\ell(t,\hat\xi),\\
\ns\ds\ell(\hat t,\xi)\les\ell(t,\xi),\qq\ell(t,\xi)\ges\ell_0+\a|\xi|.\ea\ee
\end{taggedassumption}

\begin{remark} \rm Some of the assumptions stated above can be slightly relaxed. For example, in the spirit of  \cite{Tang-Yong 1993}, we may let $x\mapsto(b(s,x),\si(s,x))$ be of linear growth, and $x\mapsto(g(s,x),h(x))$ be of some power growth. Also, the coercivity condition in \rf{ell} can be relaxed a little.
\end{remark}

Next, we introduce admissible impulse control processes with decision lag $\d\in(0,T)$. Some relevant discussions have already been carried out in the previous section.

\begin{definition}\rm An admissible impulse control process on $[t,T]$, with decision lag $\d$ and elapsed time $r$, is defined to be of form
\bel{xi(s)}\xi(s)=\sum_{i\ges1}\x_i\chi_{[\t_i,T]}(s),\qq t\les s\les T,\ee
such that the following are true:

\ms

(i) Each $\t_i$ is an $\dbF$-stopping time with
\bel{firsttime}\t_1\ges(t+\d-r)\vee t,\qq\as,\ee
and
\bel{lag}\left\{\2n\ba{ll}
\ds\t_{i+1}\ges\t_i+\d,\qq\as,\q\hb{if $\t_{i+1}<T$},\\
\ns\ds\t_{i+1}\ges\t_i,\qq\q~~\as,\q\hb{if $\t_{i+1}=T$}.\ea\right.\ee

(ii) Each $\x_i$ is $\cF_{\t_i}$-measurable with values in $K$, and
\bel{}\dbE\(\sum_{i\ges 1}\ell(\t_i,\x_i)\)<\i.\ee
\end{definition}

We let $\sK_r[t, T]$ be the set of all the impulse control processes on $[t, T]$ with decision lag $\d$ and elapsed time $r$. The last impulse time before $t$ is always denoted by $\t_0$. Then we have
\bel{r}
r=t-\t_0 \qq \text{and} \qq \t_1 \ges (\t_0 + \d)\vee t,
\ee
which give us \rf{firsttime}, i.e.  we must wait at least $(\d-r)^+$ units of time to make the first impulse after $t$. Subsequently, \rf{lag} indicates that we may intervene on the system at any times $\t_i\in[\t_1,T)$ separated at least by the decision lag $\d$. Further, the impulse can be made at terminal time $T$ without decision lag, which is an important condition to guarantee the continuity of the value function.  Besides, $r=\d\vee t$ if there has been no impulse executed on $[0,t)$. The above indicates the dependence of the impulse control on $r$, we therefore put $r$ as subscript in $\sK_r[t,T]$. Thanks to the decision lag $\d$, for any $\x(\cd)\in\sK_r[t,T]$, there exists a finite number $\k(\x(\cd))$ with
\bel{k}\k(\x(\cd))\les\[{\frac{T}{\d}}\]+1.\ee
such that
\bel{xi(L)}\x(\cd)=\sum_{i=1}^{\k(\x(\cd))}\x_i\chi_{[\t_i,T]}(\cd).\ee
We should note that an impulse control with no impulse and with zero impulses are different due to the condition \rf{ell} for the impulse cost. It is clear that any impulse control with some zero impulses are not optimal. Hereafter, we exclude all impulse controls with some zero impulses from $\sK_r[t,T]$. On the other hand, for convenience, we will use $\xi_0(\cd)$ to denote the impulse control that does not contain any impulses and call it the {\it trivial impulse control}.

\ms

Let us first present the following result which will be useful below.

\begin{proposition}\label{state} \sl Let {\rm\ref{H1}} hold. Then for any initial triple $(t,r,x)\in\sD$ and impulse control $\xi(\cd)\in\sK_r[t,T]$, state equation \rf{state1} admits a unique solution $X(\cd)\equiv X(\cd\,;t,r,x,\xi(\cd))$. Further, if $(\hat t,\hat r,\hat x)\in\sD$ with $\hat t\in[t,T]$, $\h\xi(\cd)\in\sK_{\hat r}[\hat t,T]$, and $\h X(\cd)=X(\cd\,;\hat t,\hat r,\hat x,\h\xi(\cd))$, then for any $p\ges1$, and $s\in[\hat t,T]$,
\bel{|X-X|}\ba{ll}
\ns\ds\dbE\[\sup_{s'\in[\hat t,s]}|X(s')-\h X(s')|^p\]\\
\ns\ds\les
C\dbE\[|x-\hat x|^p+|t-\hat t|^{\frac{p}{2}}+\(\sum_{\t_i<\hat t}|\xi_i|\)^p+\sup_{s'\in[\hat t,s]}|\xi(s')-\h\xi(s')|^p\],\ea\ee
and
\bel{|X-X|*}\ba{ll}
\ns\ds\dbE|X(s)-\h X(s)|^p\les C\dbE\[|x-\hat x|^p+|t-\hat t|^{\frac{p}{2}}+\(\sum_{\t_i<\hat t}|\xi_i|\)^p\\
\ns\ds\qq\qq\qq\qq\q+\(\int_{\hat t}^s|\xi(\t)-\h\xi(\t)|^2d\t\)^{\frac{p}{2}}+|\xi(s)-\h\xi(s)|^p\].\ea\ee
Hereafter, $C$ stands for a generic constant which could be different from line to line.

\end{proposition}

\begin{proof} \rm First of all, for any $(t,r,x)\in\sD$ and $\xi(\cd)\in\sK_r[t,T]$, by a standard argument making use of the contraction mapping theorem, we know that the state equation \rf{state1} admits a unique solution $X(\cd)\equiv X(\cd\,;t,r,x,\xi(\cd))$. Then for any $\hat x\in\dbR^n$, we have
$$\ba{ll}
\ns\ds\dbE\[\sup_{s'\in[t,s)}|X(s')-\hat x|^p\]\les4^{p-1}\dbE\[|x-\hat x|^p+\(\int_t^s|b(\t,X(\t))|d\t\)^p\\
\ns\ds\qq\qq\qq\qq\qq+\sup_{s'\in[t,s)}\Big|\int_t^{s'}\si(\t,X(\t))dW(\t)\Big|^p
+\sup_{s'\in[t,s]}|\xi(s)|^p\]\\
\ns\ds\les4^{p-1}\[|x-\hat x|^p\1n+L^p(s-t)^p\1n+C\dbE\(\int_t^s\1n|\si(\t,X(\t))|^2d\t\)^{\frac{p}{2}}\1n+\dbE\(\sum_{\t_i<s}|\xi_i|\)^p\]\\
\ns\ds\les C\[|x-\hat x|^p+(s-t)^{\frac{p}{2}}+\dbE\(\sum_{\t_i<s}|\xi_i|\)^p\].\ea$$
In particular, for any $\hat t\in(t,T)$,
$$\dbE|X(\hat t-)-\hat x|^p\les C\[|x-\hat x|^p+|t-\hat t|^{\frac{p}{2}}+\dbE\(\sum_{\t_i<\hat t}|\xi_i|\)^p\].$$
Next, for $(t,r,x),(\hat t,\hat r,\hat x)\in\sD$ with $0\les t<\hat t$, and $\xi(\cd)\in\sK_r[t,T]$, $\h\xi(\cd)\in\sK_{\hat r}[\hat t,T]$, let $X(\cd)$ and $\h X(\cd)$ be the corresponding solutions of the state equation \rf{state1}.
Denote
$$\ba{ll}
\ns\ds\eta(s)=\xi(s)-\h\xi(s),\qq Y(s)=X(s)-\h X(s)-\eta(s),\\
\ns\ds B(s)=\frac{[b(s,X(s))-b(s,\h X(s))][X(s)-\h X(s)]^\top}{|X(s)-\h X(s)|^2}\chi_{\{X(s)\ne\h X(s)\}},\\
\ns\ds\Si(s)=\frac{[\si(s,X(s))-\si(s,\h X(s))][X(s)-\h X(s)]^\top}{|X(s)-\h X(s)|^2}\chi_{\{X(s)\ne\h X(s)\}}.\ea$$
Then $B(\cd)$ and $\Si(\cd)$ are bounded and
$$Y(s)\1n=\1n Y(\hat t\,)+\1n\int_{\hat t}^s\1n B(\t)\big[Y(\t)+\eta(\t)\big]d\t+\1n\int_{\hat t}^s\1n\Si(\t)\big[Y(\t)+\eta(\t)\big]dW(\t),\q\3n s\1n\in\1n[\hat t,T].$$
This is equivalent to the following:
$$\left\{\2n\ba{ll}
\ds dY(s)=B(s)[Y(s)+\eta(s)]ds+\Si(s)[Y(s)+\eta(s)]dW(s),\qq s\in[\hat t,T],\\
\ns\ds Y(\hat t\,)=X(\hat t-)-\hat x-[\xi(\hat t\,)-\h\xi(\hat t\,)].\ea\right.$$
Hence, by \ref{H1}, and a standard argument for SDEs, we have
\bel{|Y|}\dbE\[\sup_{s'\in[\hat t,s]}|Y(s')|^p\]\2n\les\1n C\dbE\[|Y(\hat t\,)|^p+\(\int_{\hat t}^s\2n|\eta(\t)|d\t\)^p\3n+\1n\(\int_{\hat t}^s\2n|\eta(\t)|^2d\t\)^{p\over2}\].\ee
Consequently,
$$\ba{ll}
\ns\ds\dbE\[\1n\sup_{s'\in[\hat t,s]}\1n|X(s')-\h X(s')|^p\]\1n\les\1n2^{p-1}\dbE\[\1n\sup_{s'\in[\hat t,s]}\1n|Y(s')|^p\1n+\2n\sup_{s'\in[\hat t,s]}\1n|\xi(s')-\h\xi(s')|^p\]\\
\ns\ds\les C\dbE\Big\{|X(\hat t-)-\hat x|^2+|\xi(\hat t\,)-\h\xi(\hat t\,)|^p+\(\int_{\hat t}^s|\xi(\t)-\h\xi(\t)|^2d\t\)^{p\over2}\\
\ns\ds\qq\qq\qq\qq\qq\qq\qq\qq\qq\qq+\sup_{s'\in[\hat t,s]}|\xi(s')-\h\xi(s')|^p\]\Big\}\\
\ns\ds\les C\dbE\[|x-\hat x|^p+|t-\hat t|^{p\over2}+\(\sum_{\t_i<\hat t}|\xi_i|\)^p+\sup_{s'\in[\hat t,s]}|\xi(s')-\h\xi(s')|^p\].\ea$$
This gives \rf{|X-X|}. Also, from \rf{|Y|}, we have
$$\dbE|Y(s)|^p\les C\dbE\[|Y(\hat t)|^p+\(\int_{\hat t}^s|\eta(\t)|d\t\)^p+\(\int_{\hat t}^s|\eta(\t)|^2d\t\)^{p\over2}\],$$
which implies
$$\ba{ll}
\ns\ds\dbE|X(s)-\h X(s)|^p\les2^{p-1}\dbE\[|Y(s)|^p+|\xi(s)-\h\xi(s)|^p\]\\
\ns\ds\les \1n C\dbE\[|x\1n-\1n\hat x|^p\2n+\1n|t\1n-\1n\hat t|^{p\over2}\2n+\1n\(\1n\sum_{\t_i<\hat t}|\xi_i|\)^p\2n+\1n\(\1n\int_{\hat t}^s\1n|\xi(\t)\1n-\1n\h\xi(\t)|^2d\t\)^{p\over2}\3n+\1n|\xi(s)\1n-\1n\h\xi(s)|^p\].\ea$$
This completes the proof.
\end{proof}

\ms

From the above, we see that under \ref{H1}--\ref{H2}, for any initial triple $(t,r,x)\in\sD$ and $\xi(\cd)\in\sK_r[t,T]$, the cost functional \rf{J1} is well-defined. Then Problem (IC) can be stated as in the previous section, and the value function $V:\sD\to\dbR$ is well-defined by \rf{bar xi}. The following result is concerned with some basic properties of the value function.

\begin{theorem}\label{properties} \sl Let {\rm\ref{H1}--\ref{H2}} hold. Then
\bel{S1}|V(t,r,x)|\les L(T+1),\qq\forall(t,r,x)\in\sD,\ee
and
\bel{S2}\ba{ll}
\ns\ds|V(t,r,x)-V(\hat t,\hat r,\hat x)|\les C\big(|t-\hat t|^{1\over2}+|r\land\d
-\hat r\land\d|^{1\over2}+|x-\hat x|\big),\\
\ns\ds\qq\qq\qq\qq\qq\qq\qq\qq\forall(t,r,x),(\hat t,\hat r,\hat x)\in\sD.\ea\ee

\end{theorem}

The proof is lengthy and technical, which will be split into several lemmas. First, we have the following lemma which gives the boundedness of the value function $V(\cd\,,\cd\,,\cd)$ as well as the Lipschitz continuity of $x\mapsto V(t,r,x)$.

\begin{lemma}\label{L2} \sl Let {\rm\ref{H1}--\ref{H2}} hold. Then \rf{S1} holds and there exists a constant $C>0$ such that
\bel{|V-V|<|x-x|}|V(t,r,x)-V(t,r,\hat x)|\les C|x-\hat x|,\qq\forall(t,r,x),(t,r,\hat x)\in\sD.\ee

\end{lemma}

\begin{proof} \rm First of all, recalling the trivial impulse control $\x_0(\cd)$. By the definition of $V(t,r,x)$ and {\rm\ref{H1}--\ref{H2}}, we know that
$$\ba{ll}
\ns\ds V(t,r,x)\les J(t,r,x;\xi_0(\cd))\\
\ns\ds=\dbE\Big\{\int_t^Tg\big(s,X(s;t,r,x,\xi_0(\cd))\big)ds+ h\big(X(T;t,r,x,\xi_0(\cd))\big)\Big\}\les L(T+1).\ea$$
On the other hand, since the impulse cost $\ell(\cd)$ is positive valued, we have that, for any $\xi(\cd)\in\sK_r[t,T]$,
$$\ba{ll}
\ds J(t,r,x;\xi(\cd))\ges\dbE\Big\{\int_t^Tg\big(s,X(s;t,r,x,\xi(\cd))\big)ds
+h\big(X(T;t,r,x,\xi(\cd))\big)\Big\}\\
\ns\ds \qq\qq\qq \ges-L(T+1).\ea$$
Thus, \rf{S1} follows.

\ms

Next, for any $p\ges1$, by \rf{|X-X|} with $(t,r)=(\hat t,\hat r)$ and $\xi(\cd)=\h\xi(\cd)$, one has
$$\dbE\[\sup_{s\in[t,T]}|X(s)-\h X(s)|^p\]\les C|x-\hat x|^p.$$
Consequently,
$$\ba{ll}
\ns\ds|J(t,r,x;\xi(\cd))-J(t,r,\hat x;\xi(\cd))|\\
\ns\ds\les\dbE\Big\{\int_t^T\1n|g(s,X(s))-g(s,\h X(s))|ds+|h(X(T))-h(\h X(T))|\Big\}\les C|x-\hat x|.\ea$$
Then \rf{|V-V|<|x-x|} follows. \end{proof}

\ms

Now, let us make an observation. For any $(t,r,x)\in\sD$, and any $\xi(\cd)\in\sK_r[t,T]$ of form \rf{xi(L)}, we have
$$\ba{ll}
\ns\ds J(t,r,x;\xi(\cd))=\dbE\[\int_t^Tg(s,X(s))ds+h(X(T))+\sum_{i=1}^{\k(\xi(\cd))}\ell(\t_i,\xi_i)\]\\
\ns\ds\ges-L(T-t+1)+\dbE\(\sum_{i=1}^{\k(\xi(\cd))}
\big[\ell_0+\a|\xi_i|\big]\)\ges-L(T+1)+\a\max_{i\ges1}\dbE|\xi_i|.\ea$$
Hence,
\bel{}\a\max_{i\ges1}\dbE|\xi_i|\les J(t,r,x;\xi(\cd))+L(T+1).\ee
Consequently, taking into account \rf{S1}, we see that there exists an absolute constant $C_0$ such that if an impulse control $\xi(\cd)$ of form \rf{xi(L)} satisfying
$$\dbE\(\max_{i\ges1}|\xi_i|\)>C_0,$$
then it must be not optimal. Hence, if we set (recall \rf{k})
\bel{K_r^0}\ba{ll}
\ns\ds\sK^0_r[t,T]\1n=\1n\Big\{\xi(\cd)=\3n\sum_{i=1}^{\k(\xi(\cd))}\3n\xi_i\k_{[\t_i,T]}
(\cd)\in\sK_r[t,T]\bigm|\\
\ns\ds\qq\qq\qq\k(\xi(\cd))\les\[{T\over\d}\]+1,\;\dbE|\xi_i|\les C_0,\q1\les i\les\k(\xi(\cd))\Big\},\ea\ee
then
\bel{}V(t,r,x)=\inf_{\xi(\cd)\in\sK^0_r[t,T]}J(t,r,x;\xi(\cd)).\ee
Note that, similar to \rf{K}, we also have
\bel{K^0}\ba{ll}
\ns\ds\sK^0_{\hat r}[t,T]\subseteq\sK_r[t,T]\subseteq\sK^0_\d[t,T]=\sK_{r'}[t,T],\\
\ns\ds\qq\qq\qq\qq\qq\qq\forall\,0\les\hat r
\les r\les\d\les r'\les T.\ea\ee
\ms

Now we are ready to prove the ${1\over2}$-H\"older continuity of $t\mapsto V(t,r,x)$.

\begin{lemma}\label{vt} \sl Let {\rm\ref{H1}--\ref{H2}} hold. Then there exists a constant $C>0$ such that
\bel{vte}\ba{ll}
\ns\ds|V(t,r,x)-V(\hat t,r,x)|\les C|t-\hat t|^{1\over2},\\
\ns\ds\qq\qq\qq\qq\forall(t,r,x),(\hat t,r, x)\in\sD,\q|\hat t-t|\les\d.\ea\ee
\end{lemma}

\begin{proof} Let $\hat t\in(t,T]$. For any $\e>0$, let $\h\xi(\cd)\in\sK_r^0[\hat t,T]$ such that
$$V(\hat t,r,x)\les J(\hat t,r,x;\h\xi(\cd))<V(\hat t,r,x)+\e.$$
We extend $\h\xi(\cd)$ from $[\hat t,T]$ to $[t,T]$ by letting
$$\xi(s)=\left\{\2n\ba{ll}
\ds0,\qq\q s\in[t,\hat t\,),\\
\ns\ds\h\xi(s),\qq s\in[\hat t,T].\ea\right.$$
Namely, $\xi(\cd)$ does not have impulses on $[t,\hat t\,)$. Let $X(\cd)\equiv X(\cd\,;t,r,x,\xi(\cd))$ and $\h X(\cd)\equiv X(\cd\,;\hat t,r,$ $ x,\xi(\cd))$, by \rf{|X-X|},  there exists a constant $C>0$ such that
$$\dbE\[\sup_{s\in[\hat{t},T]}|X(s)-\h X(s)|\]\les C|t-\hat t|^{\frac{1}{2}}.$$
Then
$$\ba{ll}
\ns\ds V(t,r,x)-V(\hat t,r,x)-\e\les J(t,r,x;\x(\cd))-J(\hat t,r,x;\h\xi(\cd)) \\
\ns\ds\1n\les\dbE\[|h(X(T))\1n-\1n h(\h X(T))|\1n+\2n\int_t^{\hat t}\2n|g( s,X(s))|ds \1n+\3n \int_{\hat{t}}^T\3n|g(s,X(s))\1n-\1n g(s,\h X(s))|ds\]\\
\ns\ds\les L|t-\hat t|+L(1+T)C|t-\hat t|^{1\over2},\ea$$
which implies that
$$V(t,r,x)- V(\hat t,r,x) \les C|t-\hat t|^{1\over2}.$$
Conversely, for any $\e>0$, there exists
$$\x(\cd)=\sum_{i=1}^k\x_i\chi_{[\t_i,T]}(\cd)\in\sK^0_r[t,T],$$
with $k=\k(\xi(\cd))\les\big[{T\over\d}\big]+1$ and $\dbE|\xi_i|\les C_0$ for any $1\les i\les k$ such that
$$V(t,r,x)\ges J(t,r,x;\xi(\cd))-\e.$$
Define
$$\h\xi(\cd)=\sum_{i=1}^k\x_i\chi_{[\hat\t_i,T]}(\cd)\in\sK^0_r[\hat t,T]$$
with
$$\hat\t_i=(\t_i+\hat t-t)\land T,\qq i\ges1.$$
Clearly, $\h\xi(\cd)$ is nothing but the impulse control obtained from $\x(\cd)$ by moving all the impulses at instant $\t_i$ to the (possibly later) instant $(\t_i+\hat t-t)\land T$. Let us partition the interval $I=[\hat t,T]$ as $I=I_1 \cup I_2$, with
$$I_1=[\hat t,\t_1)\cup\Big(\bigcup_{i=1}^{k-1}[\hat\t_i,\t_{i+1})\Big)\cup[\hat\t_k,T], \qq I_2=\bigcup_{i=1}^k[\hat t\vee\t_i,\hat\t_i).$$
A careful observation tells us that
\bel{xi-xi}\xi(s)-\h\xi(s)=\sum_{i=1}^k\xi_i\chi_{[\hat t\vee\t_i,\hat\t_i)}(s),\qq\xi(T)-\h\xi(T)=0.\ee
Namely, $\xi(\cd)$ and $\h\xi(\cd)$ are different only on $I_2$. Also, we note that
\bel{t>t}0\les\hat\t_i-\t_i\les\hat t-t,\qq i\ges1.\ee
Let $X(\cd)\equiv X(\cd\,;t,r,x,\xi(\cd))$ and $\h X(\cd)\equiv X(\cd\,;\hat t,r, x,\h\xi(\cd))$, by \rf{|X-X|}, we have
$$\dbE\[\sup_{s\in[\hat t,T]}|X(s)-\h X(s)|\]\les C\(|t-\hat t|^{1\over2}
+\sup_{s\in I_2}|\xi(s)|\),$$
and noting \rf{K_r^0},
$$\dbE|X(T)-\h X(T)|\les C\dbE\[|t-\hat t|^{1\over2}+\(\int_{I_2}|\xi(\t)-\h\xi(\t)|d\t\)^{1\over2}\]\les C|t-\hat t|^{1\over2}.$$
Then, by \ref{H1}--\ref{H2} and the definition of $\sK_r^0[t,T]$, we see that
$$\ba{ll}
\ns\ds V(t,r,x)+\e-V(\hat t,r,x)\ges J(t,r,x;\xi(\cd))-J(\hat t,r,x;\h\xi(\cd))\\
\ns\ds=\dbE\Big\{\int_t^{\hat t}g(\t,X(\t))d\t+\int_{\hat t}^T\big[g(\t,X(\t))-g(\t,
\h X(\t))\big]d\t\\
\ns\ds\qq\qq+h(X(T))-h(\h X(T))+\sum_{i=1}^k\big[\ell(\t_i,\xi_i)-\ell(\hat\t_i,\xi_i)\big]\Big\}\\
\ns\ds\ges-L(\hat t-t)-C|t-\hat t|^{1\over2}-C\sum_{i=1}^k(\hat\t_i-\t_i)\ges-C|t-\hat t|^{1\over2}.\ea$$
Hence, \rf{vte} follows.
\end{proof}

Finally, we prove the continuity of the value function with respect to $r$.

\begin{lemma}\label{vr} \sl Let {\rm\ref{H1}--\ref{H2}} hold. Then there exists a constant $C>0$ such that for every $(t,r,x),(t,\hat r,x)\in\sD$,
\bel{vr7}\left\{\ba{ll}
\ns0\les V(t,r,x)-V(t,\hat r,x)\les C|r-\hat r|^{1\over2},\qq\hb{~if~}r<\hat r<\d,\\
\ns\ds0\les V(t,r,x)-V(t,\hat r,x)\les C|r-\d|^{1\over2},\qq\hb{~if~}r<\d\les\hat r,\\
\ns\ds0=V(t,r,x)-V(t,\hat r,x),\qq\qq\qq\q\;\hb{~if~}\d\les r<\hat r.\ea\right.\ee
\end{lemma}

\begin{proof} First, we consider the case $r<\hat r\les\d$. Since $\sK_r^0[t,T]\subseteq\sK_{\hat r}^0[t,T]$, we have
$$V(t,\hat r,x)\les V(t,r,x).$$
Conversely, for any $\e>0$, let
$$\h\xi(\cd)=\sum_{i=1}^k\h\xi_i\chi_{[\hat{\t}_i,T]}(\cd)\in\sK^0_{\hat r}[t,T],$$
with $k\les\big[{T\over\d}\big]+1$, and $\dbE|\h\xi_i|\les C_0$ for any $1\les i\les k$ (see \rf{K_r^0}) such that
$$V(t,\hat r,x)\les J(t,\hat r,x;\h\xi(\cd))<V(t,\hat r,x)+\e.$$
Define
$$\xi(\cd)=\sum_{i=1}^k\h\xi_i\chi_{[\t_i,T]}(\cd)\in\sK_r^0[t,T],$$
with
$$\t_i=(\hat\t_i+\hat r-r)\land T\ges\hat\t_i,\qq 1\les i\les k.$$
Clearly, $\x(\cd)$ is nothing but the impulse control obtained from $\h\xi(\cd)$ by moving the impulse at instant $\hat\t_i$ to the (corresponding) later instant $(\hat\t_i+\hat r-r)\land T$. Keep in mind that due to the decision lag, at $\hat\t_i$, only one impulse appears. Then we write $[t,T]=U_1\cup U_2$ with
$$U_1=[t,\hat\t_1)\cup\Big(\bigcup_{i=1}^{k-1}[\t_i,\hat\t_{i+1})\Big)\cup [\hat\t_k,T],
\qq U_2=\bigcup_{i=1}^k[\hat\t_i,\t_i).$$
Similar to \rf{xi-xi}, one has
\bel{xi-xi*}\h\xi(s)-\xi(s)=\sum_{i=1}^k\h\xi_i\chi_{[\hat \t_i,\t_i)}(s),\qq\h\xi(T)-\xi(T)=0.\ee
Also,
\bel{r-r}0\les\t_i-\hat\t_i\les\hat r-r,\qq i\ges1.\ee
Now, let $X(\cd)\equiv X(\cd\,;t,r,x,\xi(\cd))$ and $\h X(\cd)\equiv X(\cd\,;t,\hat r, x,\h\xi(\cd))$. Similar to \rf{|X-X|}, we have
$$\ba{ll}
\ns\ds\dbE\[\1n\sup_{s\in[t,T]}\1n|X(s)-\h X(s)|\]\1n\les C\dbE\[\sup_{s\in U_2}|\xi(s)-\h\xi(s)|+\(\1n\int_{U_2}\1n|\xi(\t)-\h\xi(\t)|d\t\)^{1\over2}\]\\
\ns\ds\les C\dbE\[\sup_{s\in U_2}|\xi(s)-\h\xi(s)|+|r-\hat r|^{1\over2}\]\ea$$
and, again noting \rf{K_r^0}, as well as \rf{r-r},
$$\dbE|X(T)-\h X(T)|\les C\dbE\(\int_{U_2}|\xi(\t)-\h\xi(\t)|d\t\)^{1\over2}\les C|r-\hat r|^{1\over2}.$$
Then, by \ref{H1}--\ref{H2} and the definition of $\sK_r^0[t,T]$, we see that
$$\ba{ll}
\ns\ds V(t,\hat r,x)+\e-V(t,r,x)\ges J(t,\hat r,x;\h\xi(\cd))-J(t,r,x;\xi(\cd))\\
\ns\ds=\dbE\Big\{\int_t^T\big[g(\t,\h X(\t))-g(\t,X(\t))\big]d\t+h(\h X(T))-h(X(T))\\
\ns\ds\qq\qq\qq\qq\qq\qq\qq+\sum_{i=1}^k\big[\ell(\hat\t_i,\h\xi_i)-\ell(\t_i,\h\xi_i)\big]\Big\}\\
\ns\ds\ges-L\dbE\[\int_t^T|X(\t)-\h X(\t)|d\t+|X(T)-\h X(T)|\]\\
\ns\ds\ges-C\dbE\[\int_{U_2}|\xi(\t)-\h\xi(\t)|d\t+|r-\hat r|^{1\over2}\]\ges-C|r-\hat r|^{1\over2}.\ea$$
This proves the first case.

\ms

Next, we look at the third case: $\d\les r<\hat r$. By \rf{K^0}, we have $\sK_r^0[t,T]=\sK_{\hat r}^0[t,T]$. Thus,
\bel{v0}
V(t,r,x)=V(t,\hat r,x).\ee
Finally, for the second case: $r<\d\les\hat r$, we have
$$|V(t,r,x)-V(t,\hat r,x)|=|V(t,r,x)-V(t,\d,x)|\les C|t-\d|^{1\over2}.$$
This completes the proof.
\end{proof}

Note that \rf{vr7} admits the following compact form:
$$|V(t,r,x)-V(t,\hat r,x)|\les C|r\land\d-\hat r\land\d|^{1\over2}.$$
Hence, combining the above three lemmas, we obtain a proof of Theorem \ref{properties}. From the above, we also see that for any $(t,x)\in[0,T]\times\dbR^n$, $V$ decreases with respect to $r$ when $r\in[0,\d)$ and keeps as a constant with respect to $r$ when $r\in[\d,T]$, i.e.
\bel{V(d)}V(t,r,x)=V(t,\d,x)\equiv V^0(t,x),\qq\forall(t,r,x)\in[0,T]\times[\d,T]\times\dbR^n.\ee
which means that the optimal value of cost will be smaller if the time we have to wait is shorter and the optimal value of cost will be the same if we don't have to wait at all.
We let
\bel{hC}\ba{ll}
\ns\ds \h C(\sD)=\Big\{v\in C(\sD)\bigm|v(t,r,x)\les L(T+1), (t,r,x)\in \sD, \\
\ns\ds\qq\qq\qq v(t,r,x)=v(t,\d,x), (t,r,x)\in[0,T]\times[\d,T]\times\dbR^n\Big\}\ea\ee
which is a class of functions that the value function $V(\cd\,,\cd\,,\cd)$ belongs to.
The following result will be used below.

\begin{corollary}\label{N0} \sl Let {\rm(H1)--(H2)} hold. Let
\bel{N[V]}N[V](t,0,x)=\inf_{\xi\in K}\big\{V(t,0,x+\xi)+\ell(t,\xi)\big\},\qq(t,x)\in[0,T]\times\dbR^n.\ee
Then $(t,x)\mapsto N[V](t,0,x)$ is continuous.

\end{corollary}

\begin{proof} \rm By the coercivity condition in \rf{ell} for $\xi\mapsto\ell(t,\x)$, we see that for any $R>0$, there exists a $C=C_R>0$ such that
$$N[V](t,0,x)=\inf_{\xi\in K\cap B(0,C_R)}\big\{V(t,0,x+\xi)+\ell(t,\xi)\big\},$$
where $B(0,C_R)$ is the closed ball centered at 0 with radius $C_R$. Now, it is clear that $(t,x)\mapsto V(t,0,x+\xi)+\ell(t,\xi)$ is uniformly continuous on $[0,T]
\times B(0,R)$ with the continuity uniform in $\xi\in K\cap B(0,C_R)$. Hence,
our conclusion follows. \end{proof}

\section{Dynamic Programming Principle and HJB Equation}\label{Sec:Bellman's Dynamic Programming}

In this section, we first establish a Bellman dynamic programming principle. Then we derive the corresponding HJB equation for the value function. The cases $r\in[0,\d)$ and $r\in[\d,T)$ will be discussed separately. Note that for the initial triplet $(t,r,x)\in[0,T]\times[0,T]\times\dbR^n$, if $r\in[0,\d)$, we must wait at least $\d-r$ to make the first impulse after time $t$, which means that there is no impulse during interval $[t,t+\d-r)$; if $r\in[\d,T]$, we can make an impulse, say, $\xi$, to the system at any time, say, $\t_1\ges t$. After such an impulse is made, the new initial triplet becomes $(\t_1,0,X(\t_1-)+\xi)$, with the elapsed time $r=0<\d$. Hence, we have the following result.

\begin{theorem}\label{max1} \sl {\rm(i)} Let $(t,r,x)\in[0,T)\times[0,\d)\times\dbR^n$, then for any $s\in[t,t+\d-r)$,
\bel{DPP1}V(t,r,x)= \dbE\Big\{V\big(s,r+s-t, X^0(s;t,x)\big)+\int_t^s g\big(\t, X^0(\t;t,x)\big)d\t\Big\},\ee
where $X^0(\cd\,;t,x)\equiv X^0(\cd)$ is the solution to the following:
\bel{y0}X^0(s)=x+\int_t^sb(\t,X^0(\t))d\t+\int_t^s\si(\t,X^0(\t))dW(\t),\q s\in[t,T].\ee

{\rm(ii)} Let $(t,r,x)\in[0,T)\times[\d,T)\times\dbR^n$, then
\bel{V^0<}V^0(t,x)\les\dbE\Big\{V^0\big(s,X^0(s;t,x)\big)+\int_t^sg\big(\t,X^0(\t;t,x)\big)d\t\Big\},\q\forall s\in[t,T],\ee
and
\bel{N}V^0(t,x)\les\inf_{\x\in K}\Big\{V(t,0,x+\x)+\ell(t,\x)\Big\}\equiv N[V](t,0,x).\ee
If, at some point $(t,r,x)\in[0,T)\times[\d,T)\times\dbR^n$, a strict inequality holds in \rf{N}, then a $t_0\in (t,T]$ exists such that
\bel{DPP2}V^0(t,x)=\dbE\Big\{V^0\big(s,X^0(s;t,x)\big)+\int_t^sg\big(\t,X^0(\t;t,x)\big)d\t\Big\},\q\forall s\in[t, t_0).\ee

\end{theorem}

\begin{proof} (i) Fix any $s\in[t,t+\d-r)$. For any $\xi(\cd)\in\sK^0_{r+s-t}[s,T]$, we may naturally extend it to a $\h\x(\cd)\in\sK^0_r[t,T]$ by not making impulses on $[t,s)$, followed
by $\xi(\cd)$. Then
$$\ba{ll}
\ns\ds V(t,r,x)\les J(t,r,x;\h\xi(\cd))\\
\ns\ds\qq\qq=\dbE\Big\{\int_t^s g(\t,X^0(\t;t,x))d\t+ J(s,r+ s- t,X^0(s;t,x);\xi(\cd))
\Big\}.\ea$$
Thus, one has
 $$V(t,r,x)\les\dbE\Big\{\int_t^sg(\t,X^0(\t;t,x))d\t+V(s,r+s-t,X^0(s;t,x))
\Big\}.$$
On the other hand, for any $\e>0$, there exists a
$$\x(\cd)= \sum\limits_{i=1}^k \x_i \chi_{[\t_i,T]}(\cd) \in \sK_r^0[t,T]$$
such that
$$V(t,r,x)+\e \ges J(t,r,x;\x(\cdot)).$$
Note that $\t_1\ges t+\d-r$. Let $s\in [t,t+\d-r)$ and
$$\h\x(\cd) = \x(\cd)|_{[s,T]}=\sum_{i=1}^k\x_i\chi_{[\t_i,T]}(\cd)\in \sK_{r+s-t}^0[s,T].$$
Then
$$ \ba{ll}
V(t,r,x)+\e  \ges J(t,r,x;\x(\cdot))\\
\ns\ds=\dbE\Big\{\int_t^Tg\big(\t,X(\t;t,x,\x(\cd))\big)d\t+\sum_{i=1}^k\ell
(\t_i,\x_i)+h\big(X(T;t,x,\x(\cd))\big)\Big\}\\
\ns\ds=\dbE\Big\{\int_t^sg\big(\t,X(\t;t,x,\x(\cd))\big)d\t+\dbE\[\int_s^T
g\big(\t,X(\t;t,x,\x(\cd))\big)d\t\\
\ns\ds\qq\qq\qq\qq\qq\qq\qq+\sum_{i=1}^k\ell(\t_i,\x_i)+h\big( X(T;t,x,\x(\cd))\big)\bigm|\cF_s\]\Big\}\ea$$
$$\ba{ll}
\ns\ds=\dbE\Big\{\int_t^sg\big(\t,X^0(\t;t,x)\big)d\t+\dbE\[\int_s^Tg\big(\t,
X(\t;s,X^0(s;t,x),\h\x(\cd))\big)d\t\\
\ns\ds\qq\qq\qq\qq\qq+\sum_{i=1}^k\ell(\t_i,\x_i)+ h\big(X(T;s,X^0(s;t,x),\h\x(\cd))\big)\bigm|\cF_s\]\Big\}\\
\ns\ds=\dbE\Big\{\int_t^sg\big(\t,X^0(\t;t,x)\big)d\t+J\big(s,r+s-t,X^0(s;t,x);
\h\x(\cd)\big)\Big\}\\
\ns\ds\ges\dbE\Big\{\int_t^sg\big(\t,X^0(\t;t,x)\big)d\t+V\big(s,r+s-t,
X^0(s;t,x)\big)\Big\}.\ea$$
Sending $\e\to0$, we obtain the other direction of the inequality. This completes the proof of (i).

\ms

(ii) First of all, for $(t,r,x)\in[0,T)\times[\d,T)\times\dbR^n$, and $s\in[t,T]$, take any $\h\x(\cd)\in\sK_{s-t+r}^0[s,T]$, we trivially extend it to $\bar\x(\cd)\in\sK_r^0[t,T]$ by making no impulses on $[t,s)$. One has
$$\ba{ll}
\ns\ds V^0(t,x)\equiv V(t,r,x)\les J(t,r,x;\bar\xi(\cd))\\
\ns\ds\qq\q~=\dbE\[\int_t^sg(\t,X^0(\t;t,x))d\t+J\big(s,X^0(s;t,x);\h\x(\cd)\big)\].\ea$$
Hence, by taking infimum over $\h\x(\cd)\in\sK_r^0[s,T]$, we obtain (noting $s-t+r>\d$)
$$\ba{ll}
\ns\ds V^0(t,x)\1n\equiv \1n V(t,r,x)\1n\les\1n\dbE\[\int_t^s \1n g(\t,X^0(\t;t,x))d\t \1n + \1n V(s,s-t+r,X^0(s;t,x))\]\\
\ns\ds\qq\qq=\dbE\[\int_t^sg(\t,X^0(\t;t,x))d\t+V^0(s,X^0(s;t,x))\],\ea$$
which gives \rf{V^0<}. Next, for any $\h\x(\cd)\in\sK^0_{s-t+r}[t,T]$, we construct
$$\x(\cd) = \x \chi_{[t,T]}(\cd)+\h\x(\cd),$$
which is the impulse control that has an impulse $\x$ at instant $t$, followed by $\h{\x}(\cd)$. Then one has $\x(\cd) \in \sK^0_{r}[t,T]=\sK^0_{\d}[t,T]$ with $r \ges \d$. Consequently,
$$V^0(t,x)\equiv V(t,r,x)=V(t,\d,x)\les J\big(t,\d,x; \x(\cd)\big)=J\big(t,0,x+\x;\h\x(\cd)\big)+\ell(t,\x).$$
Since $\h\x(\cd)$ is arbitrary, one has
$$V^0(t,x)\les V(t,0,x+\x)+\ell(t,\x),\qq\forall\x\in K,$$
which leads to \rf{N}. Suppose a strict inequality holds in \rf{N} at some point $(t,r,x)\in\sD[\d,T]$. We claim that \rf{DPP2} holds for some $t_0\in(t,T]$, i.e., there exists a minimizing sequence $\xi^\e(\cd)\in\sK^0_r[t,T]$ such that the first
impulse time $\t_1^\e\ges t_0$. Suppose \rf{DPP2} fails, which means that for
any minimizing sequence $\x^\e(\cd)\in\sK_r^0[t,T]$, the first impulse time $\t_1^\e$ satisfies
$$\lim_{\e\to0}\t_1^\e=t,\qq \lim_{\e\to0}J(t,r,x;\xi^\e(\cd))=V(t,r,x)
\equiv V^0(t,x).$$
Consequently, we may assume that
$$\ba{ll}
\ns\ds V(t,r,x)+\e\ges J(t,r,x;\xi^\e(\cd))\\
\ns\ds=\dbE\[\int_t^{\t_1^\e}g(\t,X^0(\t;t,x))d\t
+\ell(\t_1^\e,\xi_1^\e)+J(\t_1^\e,X^0(\t_1^\e;t,x)+\xi_1^\e;\h\xi^\e(\cd))\]\\
\ns\ds\ges\dbE\[\int_t^{\t_1^\e}g(\t,X^0(\t;t,x))d\t
+\ell(\t_1^\e,\xi_1^\e)+V\big(\t_1^\e,0,X^0(\t_1^\e;t,x)+\xi_1^\e\big)\]\\
\ns\ds\ges\dbE\[\int_t^{\t_1^\e}g(\t,X^0(\t;t,x))d\t
+N[V]\big(\t_1^\e,0,X^0(\t_1^\e;t,x)\big)\].\ea$$
Sending $\e\to0$, using the continuity of $(t,x)\mapsto N[V](t,0,x)$, we obtain
$$V^0(t,x)=V(t,r,x)\ges N[V](t,0,x),$$
which is a contradiction, proving \rf{DPP2}. \end{proof}

Now let us introduce the following Hamiltonian:
\bel{H}\ba{ll}
\ns\ds H(t,x,p,P)=\lan b(t,x),p\ran+{1\over2}\tr{[\si(t,x)^\top P\si(t,x)]}+g(t,x),\\
\ns\ds\qq\qq\qq\qq\qq\qq\q(t,x,p,P)\in[0,T]\times\dbR^n\times\dbR^n\times\dbS^n.\ea\ee
We easily obtain the Hamilton-Jacobi-Bellman equations for our value function as follows:

\begin{theorem} \sl Suppose the value function $V(\cd\,,\cd\,,\cd)$ is smooth.
Then with \rf{V(d)}, the following system is satisfied:
\bel{HJB1}\left\{\2n\ba{ll}
\ds V_t(t,r,x)+V_r(t,r,x)+H(t,x,V_x(t,r,x),V_{xx}(t,r,x))=0,\\
\ns\ds\qq\qq\qq\qq\qq\qq\qq\q(t,r,x)\in
[0,T)\times[0,\d)\times\dbR^n,\\
\ns\ds V(T,r,x)\1n=\1n\min\1n\Big\{\1n h(x),\inf_{\x\in K}\{ h(x\1n+\1n\x)\1n+\1n\ell(T,\x)\}\1n\Big\},~\1n(r,x)\1n\in\1n[0,\d)\1n\times\1n\dbR^n,\\
\ns\ds V(t,\d-,x)=V^0(t,x),\qq\qq\qq\qq\qq(t,x)\in[0,T]\times\dbR^n.
\ea\right.\ee
\bel{HJB2}\left\{\ba{ll}
\ds\min\Big\{V^0_t(t,x)\1n+\1n H(t,x,V^0_x(t,x),V^0_{xx}(t,x)),\\
\ns\ds\qq\qq\q N[V](t,0,x)\1n-\1n V^0(t,x)\Big\}=0,\q(t,x)\in[0,T)\times\dbR^n,\\
\ns\ds V^0(T,x)=\min\Big\{h(x),\inf_{\x\in K}\{h(x+\x)+\ell(T,\x)\}\Big\},\qq\q x\in\dbR^n.\ea\right.\ee

\end{theorem}
\begin{proof} Let us first prove that $V(\cd\,,\cd\,,\cd)$ satisfies \rf{HJB1}. Fix any $(t,r,x)\in[0,T]\times[0,\d]\times\dbR^n$, and let $X^0(\cd)$ be the state trajectory defined by \rf{y0}. By \rf{DPP1} with $s\downarrow t$ and It\^o's formula, we obtain
$$\ba{ll}
\ns\ds0={\dbE\big\{V(s,r+s-t,X^0(s))-V(t,r,x)\big\}\over s-t}+{1\over s-t}\dbE\int_t^sg(\t,X^0(\t))d\t\\
\ns\ds\q={1\over s-t}\dbE\int_t^s\big\{V_t(\t,r+\t-t,X^0(\t))+V_r(\t,r+\t-t,X^0(\t))\\
\ns\ds\qq\q+H\big(\t,X^0(\t),V_x(\t,r+\t-t,X^0(\t)),V_{xx}(\t,r+\t-t, X^0(\t))\big)\big\}d\t\\
\ns\ds\q\to V_t(t,r,x)+V_r(t,r,x)+H(t,x,V_x(t,r,x),V_{xx}(t,r,x)).\ea$$
By \rf{v0} and the continuity of $V(\cd\,,\cd\,,\cd)$ with respect to $r$, one has
$$V(t,\d-,x)=V(t,\d,x)=V^0(t,x).$$
The terminal condition at time $T$ comes from the assumption that impulse can be made at terminal without decision lag. Now, we show that $V^0(t,x)\equiv V(t,r,x)$ (for $r\in [\d,T]$) satisfies \rf{HJB2}. From Theorem \ref{max1} (ii), using It\^o's formula, we have
\bel{V_t+H>0}\2n\ba{ll}
\ns\ds0\1n\les\1n V^0_t(t,x)\1n+\1n\lan V_x^0(t,x),b(t,x)\ran\1n+\1n{1\over2}\tr\1n\big[\si(t,x)^\top\1n V^0_{xx}(t,x)\si(t,x)\big]\2n+\1n g(t,x)\\
\ns\ds\q\2n\equiv V^0_t(t,x)+H(t,x,V^0_x(t,x),V^0_{xx}(t,x)),\ea\ee
and
\bel{N-V>0}V^0(t,x)\les N[V](t,0,x).\ee
On the other hand, by the last part of Theorem \ref{max1}, we see that when a strict
inequality holds in \rf{N-V>0}, then the equality in \rf{V_t+H>0} holds. Hence, the equation in \rf{HJB2} holds. Finally, the terminal condition in \rf{HJB2} hold by definition. This completes the proof. \end{proof}

\ms

Note that \rf{HJB1} and \rf{HJB2} are coupled.  The coupling appears at the following places: The value $V(t,\d,x)$ is equal to $V^0(t,x)$, and the obstacle $N[V](t,0,x)$ depends on $V(t,0,\cd)$. We may make a comparison between our \rf{HJB1}--\rf{HJB2} and \rf{v^0}--\rf{v^1}. We see that although our \rf{HJB2} looks similar to \rf{v^0}, they are still quite different in a number of places. On the other hand, our \rf{HJB1} is not comparable with \rf{v^1} since the appearance of $V_r(t,r,x)$ in our equation. The main reason is that we have carefully taken into account of the elapsed time $r$, which was essentially overlooked in \cite{Bruder-Pham 2009}.

 To make the above easy to solve mathematically, we give an equivalent system to \rf{HJB1} and \rf{HJB2} as follows.
\begin{theorem} \sl Suppose the value function $V(\cd\,,\cd\,,\cd)$ is smooth.
Then, with \rf{S1} and \rf{V(d)}, the following system is satisfied:
\bel{HJB3}\left\{\2n\ba{ll}
\ds \min\Big\{V_t(t,r,x)+V_r(t,r,x)+ H(t,x,V_x(t,r,x),V_{xx}(t,r,x)),\\
\ns\ds\qq~\h N[V](t,r,x)\1n-\1n V(t,r,x)\1n\Big\}\2n=\1n0,\q(t,r,x)\1n\in\1n
[0,T)\1n\times\1n[0,T)\1n\times\1n\dbR^n,\\
\ns\ds V(T,r,x)=\min\Big\{h(x),\inf_{\x\in K}\{ h(x+\x)+\ell(T,\x)\}\Big\},\\
\ns\ds\qq\qq\qq\qq\qq\qq\qq(r,x)\in[0,T]\times\dbR^n,\ea\right.\ee
where
\bel{hN} \h{N}[V](t,r,x)=N[V](t,0,x)\chi_{\{r \ges \d\}} +[2L(T+1)]\chi_{\{r< \d\}}. \ee
\end{theorem}
\begin{proof}
When $r\in [0, \d)$, by \rf{S1}, we have
$$V(t,r,x) \les L(T+1) < \h{N}[V](t,r,x),$$
which implies that \rf{HJB3} is equivalent to \rf{HJB1} in $[0,T) \times [0, \d) \times \dbR^n$.

On the other hand, by \rf{V(d)} we know that $V_r(t,r,x)=0$ when $ r \in [\d, T)$. Hence \rf{HJB3} is equivalent to \rf{HJB2} in $[0,T) \times [\d, T) \times \dbR^n$, which concludes the proof.
\end{proof}
\section{Characterization of the Value Function and Construction of Optimal Control} \label{Sec:Characterization of the Value Function}

It is known that the value function $V(\cd\,,\cd\,,\cd)$ is not necessarily smooth. Thus to make the result of Theorem 3.2 rigorous, let us recall the definition of viscosity solutions. Note that $\sD=[0,T]\times[0,T]\times\dbR^n$, and recall $\h C(\sD)$ defined by \rf{hC}.

\begin{definition}\label{vcd} \rm A function $V(\cd\,,\cd\,,\cd)\in\h C(\sD)$ is called a {\it viscosity subsolution} (resp.{\it viscosity supersolution}) of \rf{HJB3} on $\sD$ if for any $\f\in C^{1,1,2}(\sD)\cap\h C(\sD)$,
\bel{V(T)>}\ba{ll}
\ns\ds V(T,r,x)\1n\les\1n\({\rm resp.}\ges\)\min\1n\Big\{h(x),\inf_{\x\in K}\1n\big[h(x+\x)+\ell(T,\x)\big]\Big\},\\
\ns\ds\qq\qq\qq\qq\qq\qq\qq\qq\forall(r,x)\in[0,T]\times\dbR^n,\ea\ee
and whenever $V-\f$ achieves a local maximum (resp. minimum) at $(t_0,r_0,x_0)$ $\in[0,T) \times[0,T) \times \dbR^n$, it holds
\bel{vc1}\ba{ll}
\ds\min\Big\{\1n \f_t(t_0,r_0,x_0)\1n+\1n\f_r(t_0,r_0,x_0)\1n+\1n H(t_0,x_0,\f_x(t_0,r_0,x_0),\f_{xx}(t_0,r_0,x_0)),\\
\ns\ds\qq\qq\qq\qq\qq\q \h N[V](t_0,r_0,x_0) -V(t_0,r_0,x_0)\1n\Big\}\ges0 ~ ({\rm resp.}\les 0).\ea\ee
A function $V(\cd\,,\cd\,,\cd)\in\h C(\sD)$ is called a {\it viscosity solution} of \rf{HJB3} if it is both a viscosity sub- and super-solution of \rf{HJB3}.
\end{definition}

The main result of this section is the following.

\begin{theorem}\label{exist} \sl Let {\rm(H1)--(H2)} hold. Then the value function $V(\cd\,,\cd\,,\cd)$ is the unique viscosity solution of \rf{HJB3} on $\sD$ satisfying \rf{S1}--\rf{S2}.
\end{theorem}

\begin{proof} Let $\f\in C^{1,1,2}(\sD)\cap\h C(\sD)$. Suppose that $V-\f$ attained a local maximum (resp. minimum) at $(t_0,r_0,x_0)\in[0,T)\times[0,T)\times\dbR^n$ and $X^0(\cd)$ be the state trajectory defined by \rf{y0}. By \rf{DPP1} with $\hat t_0\downarrow t_0$ and It\^o's formula, one has
$$\ba{ll}
\ds0\les\(\hb{resp.}\ges\){1\over\hat t_0-t_0}\dbE\Big\{V(t_0,r_0,x_0)-\f(t_0,r_0,x_0)\\
\ns\ds \qq \qq \qq -V(\hat{t}_0,r_0+\hat{t}_0-t_0, X^0(\hat{t}_0))+\f(\hat{t}_0,r_0+\hat{t}_0-t_0, X^0(\hat{t}_0))\Big\}\\
\ns\ds = \1n\frac{1}{\hat{t}_0-t_0}\dbE\Big\{\1n\int_{t_0}^{\hat{t}_0}\1n g(\t, X^0(\t))d \t -\f(t_0,r_0,x_0)+\f(\hat{t}_0,r_0+\hat{t}_0-t_0, X^0(\hat{t}_0))\Big\}\\
\ns\ds \to \f_t(t_0,r_0,x_0)+\f_r(t_0,r_0,x_0)+H(t_0,x_0,\f_x(t_0,r_0,x_0),\f_{xx}(t_0,r_0,x_0)).
\ea$$
Then one has \rf{vc1}, with the fact that
$$\h N[V](t_0,r_0,x_0) -V(t_0,r_0,x_0) \ges 0.$$

\ms

The proof of uniqueness essentially follows from the arguments in \cite{Tang-Yong 1993} and \cite{Yong-Zhou 1999}, with some suitable modifications. For readers' convenience, we put the detailed proof in the appendix.  \end{proof}

\ms

The following gives a construction of an optimal impulse control.

\begin{theorem} \sl Let $x$ be the initial state of \rf{state1}. Define the impulse control $\x(\cdot)=\sum\limits_{i \ges 1}\x_i \chi_{[\t_i, T]}$ on $[0,T]$ inductively as follows: $\t_0=-\d$,
$$\ba{ll}
\ns\ds\t_{i}=\inf\Big\{s\1n\in\1n[(\t_{i\1n-\1n 1}\1n+\d)\land T,T]\bigm|\\
\ns\ds\qq\qq\qq V(s,s-\t_{i \1n-\1n 1},\1n X^{\1n(i\1n-\1n 1)}\1n (s))\1n=\1n N[V](s,s-\t_{i\1n-\1n 1},\1n X^{\1n(i\1n-\1n 1)}\1n(s))\Big\},\ea$$
and there exists a $\x_i \in K$ such that
\bel{xi*} \ba{ll}
\ns\ds V(\t_i,\t_i-\t_{i-1},X^{(i-1)}(\t_i-0))\\
\ns\ds=V(\t_i,0,X^{(i-1)}(\t_i-0)+\x_i)+\ell(\t_i,\x_i)\\
\ns\ds= \inf_{\x \in K} \big\{ V(\t_i,0,X^{(i-1)}(\t_i-0)+\x)+\ell(\t_i,\x)\big\}, \ea \ee
where $1 \les i \les [\frac{T}{\d}]+1$ and $X^{(i-1)}(\cdot)$ is the result of applying impulse control $\x(\cdot)=\sum\limits_{m=1}^{i-1}\x_m \chi_{[\t_m, T]}$ on system \rf{state1} . Then, $\x(\cdot)=\sum\limits_{i \ges 1}\x_i \chi_{[\t_i, T]}$ is an optimal control for Problem (IC).
\end{theorem}
\begin{proof}
Let $\x(\cdot)=\sum\limits_{i=1}^k\x_i \chi_{[\t_i, T]}$ be constructed as in the theorem and $X(\cdot)$ be the state of applying $\x(\cdot)$. By the definition of $\t_i$ and $\x_i$, we see that
$$V(\t_i, 0, X^{(i)}(\t_i+0)) < N[V](\t_i, 0, X^{(i)}(\t_i+0)).$$
Then, by \autoref{max1} and \rf{xi*},
$$ \ba{ll}
\ds \q V(\t_{i},0,X^{(i)}(\t_{i}+0))\\
\ns\ds =\dbE\big\{\int_{\t_i}^{\t_{i+1}} g(s, X^{(i)}(s))d s+V(\t_{i+1}, \t_{i+1}-\t_i, X^{(i)}(\t_{i+1}-0))\big\}\\
\ns\ds  = \dbE\big\{\int_{\t_i}^{\t_{i+1}} g(s, X^{(i)}(s))d s+V(\t_{i+1}, 0, X^{(i+1)}(\t_{i+1}+0))+\ell(\t_{i+1},\x_{i+1})\big\}. \ea$$
Hence, summing from $0$ to $T$ gives
$$\ba{ll}
\ns\ds V(0,\d,x)=\dbE\Big\{\int_0^{\t_1}f(s,X(s))ds+\ell(\t_1,\x_1)+V(\t_1,0,X(\t_1+0))\Big\}\\
\ns\ds=\dbE\Big\{ \int_0^{\t_k}f(s,X(s))ds+\sum_{i=1}^k\ell(\t_i,\x_i)+V(\t_k, T-\t_k, X(\t_k+0))\Big\}\\
\ns\ds=\dbE\Big\{\int_0^{T}f(s,X(s))ds+\sum_{i=1}^k\ell(\t_i,\x_i)+h(X(T)) \Big\}=J(0,\d,x;\x(\cd)),\ea$$
which proves the optimality of $\x(\cdot)$.
\end{proof}

To conclude this paper, we look at the situation when the decision lag $\d$ approaches 0. We let $\d_{\e}$ be a sequence of decision lags such that
$$\lim_{\e \to 0}\t_{\e}=0.$$
Let \autoref{H1}--$\2n$\autoref{H2} hold. Denote $V_{\e}(\cdot,\cdot,\cdot)$ as the family of the value functions corresponding to the  decision lag $\d_{\e}$. Let $V^0(\cdot,\cdot)$ be the value function of classical impulse control problem without decision lag. Fix $\e>0$ and $(t,x) \in [0,T] \times \dbR^n$. From \autoref{properties} we see that when $r \in [\d_{\e},T]$,
$$V_{\e}(t,r,x)=V^0(t,x),$$
and when $r \in [0,\d_{\e})$,
$$0\les V_{\e}(t,r,x)-V^0(t,x) = V_{\e}(t,r,x) - V_{\e}(t,\d_{\e}, x) \les C|\t_{\e}-r|.$$
Let $\e \to 0$, we obtain
\be
\lim_{\e \to 0} V_{\e}(t,r,x)=V^0(t,x),\qq (t,r,x) \in [0,T] \times  [0,T] \times \dbR^n ,
\ee

 Hence, we conclude that the impulse control problem with decision lag agrees with the classical case when decision lag approaches 0.

\section{Conclusion}\label{Sec:Value FunctioConclusion}

This paper studies an optimal stochastic impulse control problem with a decision lag $\d>0$. The introduction of the elapsed time $r$ helps us fully understand the problem. Continuity of the value function $V(t,r,x)$ in all its arguments is proved directly, which is by no means trivial. A suitable version of the dynamic programming principle is established, which takes into account the cases $r \in [0,\d)$ and $r \in [\d,T)$ separately. The corresponding HJB equations are coupled and involve a new derivative term $V_r$. Following the standard approach to stochastic control, the value function is shown to be the unique viscosity solution to these HJB equations. An optimal impulse control is constructed from the given value  function. Morever, a  limiting  case  with  the  decision  lag  approaching $0$ is discussed,  which exactly recovers the classical impulse control problems.

\begin{footnotesize}

\end{footnotesize}

\section*{Appendix}

In this appendix, we present a proof of the uniqueness of the viscosity solution to the HJB equation \rf{HJB3}. We first prove a useful lemma.

\ms

\no\bf Lemma A.1. \sl Suppose $V(\cd\,,\cd\,,\cd)\in\h C(\sD)$ is a viscosity solution of \rf{HJB3} satisfying \rf{S1}--\rf{S2}. Then
\bel{V0N} V(t,r,x)\les \h N[V](t,r,x),\qq\forall(t,r,x)\in[0,T]\times[0,T]\times\dbR^n.\ee

\rm

\begin{proof}
It is enough to prove \rf{V0N} for all $(t,r,x) \in [0,T) \times [\d, T) \times  \dbR^n$, where $V(t,r,x)=V^0(t,x)$ and $\h N[V](t,r,x)=N[V](t,0,x)$.  We define
$$\F(s,y)=V^0(s,y)-\frac{1}{\e}(|t-s|^2+|x-y|^2), \qq (s,y) \in [0,T] \times  \dbR^n.$$
with some $\e \in (0,1]$. Then there exists a point $(s_\e, y_\e)\in [0,T] \times  \dbR^n$ such that
$$\F(s_\e, y_\e) = \max_{(s,y) \in [0,T] \times  \dbR^n} \F(s,y) \ges \F(t,x)=V^0(t,x).$$
We see easily that
$$\lim_{\e \to 0}|s_\e-t|=0, \qq \lim_{\e \to 0}|y_\e-x|=0.$$
Thus, for $\e>0$ small enough, we have $(s_\e, y_\e) \in [0,T) \times  \dbR^n$. Then, by  \autoref{vcd},
$$V^0(s_\e, y_\e) \les N[V](s_\e,0, y_\e).$$
Sending $\e \to 0$ and using the continuity of $V^0(\cd,\cd)$, we obtain \rf{V0N}.
\end{proof}

Next, inspired by \cite{Yong-Zhou 1999}, for any $v\in \h{C}(\sD)$ satisfying \rf{S1}--\rf{S2}, and any $\g\in(0,1)$, we define
\bel{v^g}\ba{ll}
\ds v^\g(t,r,x)\1n\deq\2n\sup_{(t',r',x')\in\sD}\1n\Big\{v(t',r',x')-{1\over2\g^2}
\(|t-t'|^2+|r-r'|^2+|x-x'|^2\)\Big\}, \\
\ds v_\g(t,r,x)\1n\deq\2n\inf_{(t',r',x')\in\sD}\1n\Big\{ v(t',r',x')+{1\over2\g^2}\(|t-t'|^2+|r-r'|^2+|x-x'|^2\)\Big\},\\
\ns\ds \qq\qq\qq\qq\qq\qq\qq\qq\qq\qq\qq\qq\qq(t,r,x)\in\sD.\ea\ee
Note that
$$\ba{ll}
\ns\ds v(t',r',x')-{1\over2\g^2}\(|t-t'|^2\1n+|r-r'|^2\1n+|x-x'|^2\)+{1\over2\g^2}\(|t|^2+|r|^2+|x|^2\)\\
\ns\ds=v(t',r',x')-{1\over2\g^2}\(|t'|^2+|r'|^2+|x'|^2\)+{1\over\g^2}\(t\,t'+r\,r'+\lan x,x'\ran\),\ea$$
which is a linear function of $(t,r,x)$ with parameters $(t',r',x')$. Hence, the supremum $v^\g(t,r,x)+{1\over2\g^2}\big(|t|^2+|r|^2+|x|^2\big)$ of the above with respect to $(t',r',x')$ is convex. In this case, we say that $(t,r,x)\mapsto v^\g(t,r,x)$ is semiconvex.
Likewise, $(t,r,x)\mapsto v_\g(t,r,x)$ is semiconcave. We have the following lemma concerning the functions $v^\g(\cd\,,\cd\,,\cd)$ and $v_\g(\cd\,,\cd\,,\cd)$.

\ms

\no\bf Lemma A.2. \sl {\rm(i)} Let $v\in \h{C}(\sD)$ satisfy \rf{S1}--\rf{S2}. Then the function $v^\g(\cd\,,\cd\,,\cd)$ is semiconvex, and $v_\g(\cd\,,\cd\,,\cd)$ is semiconcave, satisfying the following:
\bel{semi1}|v^\g(t,r,x)|+|v_\g(t,r,x)|\les C,\qq\forall(t,r,x)\in\sD,\g>0, \ee
\bel{semi11}\ba{ll}
\ns\ds \q|v^\g(t,r,x)-v^\g(s,u,y)|+|v_\g(t,r,x)-v_\g(s,u,y)|\\
\ns\ds \les C\big(|t-s|^{\frac{1}{2}}+|r-u|^{\frac{1}{2}}+|x-y|\big),\qq  \forall (t,r,x,s,u,y) \in \sD^2,\g>0,\ea\ee
for some constant $C>0$. Moreover, for any $(t,r,x)\in\sD$,
there exist $(\hat t,\hat r,\hat x)$, $(\bar t,\bar r,\bar x)\in\sD$ such that
\bel{semi2}\ba{ll}
\ns\ds v^\g(t,r,x)=v(\hat t,\hat r,\hat x)-{1\over2\g^2}\(|t-\hat t|^2+|r-\hat r|^2+|x-\hat x|^2\)\\
\ns\ds v_\g(t,r,x)=v(\bar t,\bar r,\bar x)+{1\over2\g^2}\(|t-\bar t|^2+|r-\bar r|^2
+|x-\bar x|^2\),\ea\ee
and for some absolute constant $C$,
\bel{semi3}{1\over2\g^2}\(|t-\hat t|^2+|r-\hat r|^2+|x-\hat x|^2\)
+{1\over2\g^2}\(|t-\bar t|^2+|r-\bar r|^2+|x-\bar x|^2\)\les C\g^{1\over2}.\ee
Consequently,
\bel{semi4}0\les v^\g(t,r,x)-v(t,r,x),~v(t,r,x)-v_\g(t,r,x)\les C\g^{1\over2}. \ee

\rm

The proof is inspired by \cite{Yong-Zhou 1999}. Because of \rf{semi4}, we call $v^\g(\cd\,,\cd\,,\cd)$ and $v_\g(\cd\,,\cd\,,\cd)$ {\it semiconvex} and {\it semiconcave approximations} of $v(\cd\,,\cd\,,\cd)$, respectively. Also, we see that
\bel{semi8} 0 \les v^{\g}(t,r,x)-v_{\g}(t,r,x)\les C\g^{1\over2}.\ee
Next, for any $(t,x,p,P)\in[0,T]\times\dbR^n\times\dbR^n\times\dbS^n$, we define
\bel{U2}\left\{\3n\ba{ll}
\ds H^\g\1n(t,x,p,\1n P)\1n\deq\3n\sup_{(t'\1n,x')\in[0,T]\times\dbR^n}\3n\big\{\1n H(t'\1n,x'\1n,p,\1n P)\1n\bigm|\1n {1\over2\g^2}(|t\1n-\1n t'|^2\1n+\1n|x-x'|^2)\1n\les\1n C\g^{1\over2}\1n\big\},\\
\ns\ds H_\g\1n(t,x,p,\1n P)\1n\deq\3n\inf_{(t'\1n,x')\in[0,T]\times\dbR^n}\3n\big\{\1n H(t'\1n,x'\1n,p,\1nP)\1n\bigm|\1n {1\over2\g^2}(|t\1n-\1n t'|^2\1n+\1n|x\1n-\1n x'|^2)\1n\les\1n C\g^{1\over2}\1n\big\},\\
\ns\ds \h N^\g[V](t,r,x)\1n\deq\3n\sup_{(t',r',x')\in\sD}\3n\big\{\h N[V](t'\1n,r'\1n,x')\1n-\1n{1\over2\g^2}(|t\1n-\1n t'|^2\2n+\1n|r\1n-\1n r'|^2\2n+\1n|x\1n-\1n x'|^2)\1n\bigm|\\
\ns\ds\qq\qq\qq\qq\qq\qq\q{1\over2\g^2}(|t\1n-\1n t'|^2\1n+\1n|r\1n-\1n r'|^2\1n+\1n|x\1n-\1n x'|^2])\1n\les\1n C\g^{1\over2}\big\},\\
\ns\ds \h N_\g[V](t,r,x)\1n\deq\3n\inf_{(t',r',x')\in\sD}\3n\big\{\h N[V](t'\1n,r'\1n,x')\1n+\1n{1\over2\g^2}(|t\1n-\1n t'|^2\2n+\1n|r\1n-\1n r'|^2\2n+\1n|x\1n-\1n x'|^2)\1n\bigm|\\
\ns\ds\qq\qq\qq\qq\qq\qq\q{1\over2\g^2}(|t\1n-\1n t'|^2\1n+\1n|r\1n-\1n r'|^2\1n+\1n|x\1n-\1n x'|^2)\1n\les\1n C\g^{1\over2}\big\},\ea\right.\ee
with $C>0$ being the constant appears in \rf{semi3}.
It is clear that one has
\bel{HN}\ba{ll}
\ds \lim_{\g\to 0}H^{\g}(t,x,p,P)=\lim_{\g\to 0}H_{\g}(t,x,p,P)=H(t,x,p,P),\\
\ns \ds \lim_{\g\to 0}\h N^{\g}[V^{\g}](t,r,x) =\lim_{\g\to 0}\h N_{\g}[V_{\g}](t,r,x)= \h N[V](t,r,x) \ea \ee
uniformly for the arguments $t,r,x,p,P$ in compact sets. Next we present the following result.

\ms

\no\bf Lemma A.3. \sl Let {\rm(H1)--(H2)} hold and $v(\cd\,,\cd\,,\cd)\in \h{C}(\sD)$ be a viscosity subsolution of \rf{HJB3}. Then, for each $\g \in (0,1)$, $v^{\g}(\cd\,,\cd\,,\cd)$ is a viscosity subsolution of the following:
\bel{vg1}\left\{\ba{ll}
\ds \min\big\{v_t+v_r+H^\g(t,x,v_x,v_{xx}), \h N^\g[v]-v^0\big\}=0,\\
\ns\ds \qq\qq\qq\qq\q\3n (t,r,x)\in[0,T)\times[0,T)\times\dbR^n,\\
\ns\ds v(T,r,x)=v^\g(T,r,x),\qq\q(r,x)\in[0,T]\times\dbR^n.\ea\right.\ee
Likewise, if $v(\cd\,,\cd\,,\cd)\in \h C(\sD) $ is a viscosity supersolution of
\rf{HJB3}, then, for each $\g\in(0,1)$, $v_\g(\cd\,,\cd\,,\cd)$ is a
viscosity supersolution of the following:
\bel{vg3}\left\{\ba{ll}
\ds \min\big\{v_t+v_r+H_\g(t,x,v_x,v_{xx}), \h N_\g[v]-v^0\big\}=0,\\
\ns\ds \qq\qq\qq\qq\q\3n (t,r,x)\in[0,T)\times[0,T)\times\dbR^n,\\
\ns\ds v(T,r,x)=v_\g(T,r,x),\qq\q(r,x)\in[0,T]\times\dbR^n.\ea\right.\ee

\rm

\begin{proof} Let us just look at $v^\g(\cd\,,\cd\,,\cd)$. Suppose $\f\in C^{1,1,2}(\sD)\cap\h C(\sD)$ such that $v^\g-\f$ attains a local maximum at $(t,r,x)\in[0,T)\times[0,T)\times\dbR^n$. Let $(\hat t,\hat r, \hat x)\in[0,T]\times[0,T]\times\dbR^n$ satisfy \rf{semi2}. Then for any $(t',r',x')\in
[0,T)\times[0,T)\times\dbR^n$ near $(t,r,x)$, one has
$$\ba{ll}
\ds v(\hat t,\hat r,\hat x)\1n-\1n\f(t,r,x)\1n=\1n v^\g(t,r,x)\1n-\1n\f(t,r,x)\1n+\1n{1\over2\g^2}
\big(|t\1n-\1n\hat t|^2\1n+\1n|r\1n-\1n\hat r|^2\1n+\1n|x\1n-\1n\hat x|^2\big)\\
\ns\ds\qq\q\ges v^\g(t',r',x')-\f(t',r',x')+{1\over2\g^2}\big(|t-\hat t|^2+|r-\hat r|^2+|x-\hat x|^2\big)\\
\ns\ds\qq\q\ges v(t'-t+\hat t,r'-r+\hat r,x'-x+\hat x)-\f(t',r',x').\ea$$
Consequently, for any $(\t,\rho,\z)\in[0,T)\times[0,t)\times\dbR^n$, near $(\hat t,\hat r,\hat x)$, by letting $t'=\t+t-\hat t$, $r'=\rho+r-\hat r$ and $x'=\z+x-\hat x$, we get
$$v(\hat t,\hat r,\hat x)-\f(t,r,x)\ges v(\t,\rho,\z)-\f(\t+t-\hat t,\rho+r-\hat r,\z+x-\hat x),$$
which means that the function $(\t,\rho,\z)\mapsto v(\t,\rho,\z)-\f(\t+t-\hat t,\rho+r-\hat r,\z+x-\hat x)$ attains a local maximum at $(\t,\rho,\z)=(\hat t,\hat r,\hat x)$. Thus, by \rf{vc1} and \rf{U2}, we obtain
$$\ba{ll}
\ds\f_t(t,r,x)+\f_r(t,r,x)+H^\g(t,x,\f_x(t,r,x),\f_{xx}(t,r,x))\\
\ns\ds\ges \f_t(t,r,x)+\f_r(t,r,x)+H(t,x,\f_x(t,r,x),\f_{xx}(t,r,x)) \ges 0.\ea$$
Moreover, by \rf{V0N}, we obtain
\be \ba{ll}
\ds v^{\g}(t,r,x)\1n=\1n \sup_{\1n(t',r',x')\in\sD}\1n\Big\{v(t',r',x')\1n-\1n{1\over2\g^2}\(|t\1n-\1n t'|^2\1n+\1n|r\1n-\1n r'|^2\1n+\1n|x\1n-\1n x'|^2
\)\Big\}\\
\ns\ds\les\sup_{(t',r',x')\in\sD}\Big\{\h N[v](t',r',x')-{1\over2\g^2}\(|t- t'|^2+|r-r'|^2+|x- x'|^2\)\bigm|\\
\ns\ds \qq\qq\qq\qq\qq\q{1\over2\g^2}\(|t- t'|^2+|r-r'|^2+|x-x'|^2\)\les C\g^{1\over2}\Big\} \\
\ns\ds = \h N^{\g}[v^\g](t,r,x).\ea \ee
Thus,
$$\ba{ll}
\ns\ds\min\Big\{\f_t(t,r,x)+\f_r(t,r,x)+H \big(t,x,\f_x(t,r,x),\f_{xx}(t,r,x)\big),\\
\ns\ds\qq\qq\qq\qq\qq\qq\qq \h N^{\g}[v^{\g}](t,r,x)\1n-\1n v^{\g}(t,r,x)\Big\}\ges0.\ea $$
This proves that $v^\g$ is a viscosity subsolution of \rf{vg1}. In a same manner, we can prove that $v_\g$  is a viscosity supersolution of \rf{vg3}. \end{proof}

Now we are ready to prove the uniqueness of the viscosity solution.

\begin{proof}  Let $V(\cd\,,\cd\,,\cd), \h V(\cd\,,\cd\,,\cd) \in \h C(\sD)$ be two viscosity solutions of \rf{HJB3} satisfying \rf{S1}--\rf{S2}. We claim that
\bel{contra}V(t,r,x)\les\h V(t,r,x),\qq\forall(t,r,x)\in\sD.\ee
We prove this by contradiction. Suppose \rf{contra} is false, then there exists a point $(\bar t,\bar r,\bar x)\in(0,T)\times(0,T)\times\dbR^n$ such that
$$2\eta\deq V(\bar t,\bar r,\bar x)-\h V(\bar t,\bar r,\bar x)>0.$$
Let $V^\g(\cd\,,\cd\,,\cd)$ and $\h V_\g(\cd\,,\cd\,,\cd)$ be the semiconvex and semiconcave approximations of $V(\cd\,,\cd\,,\cd)$ and $\h V(\cd\,,\cd\,,\cd)$, respectively. By \rf{semi4}, for all small enough $\g>0$,
\bel{u5}V^\g(\bar t,\bar r,\bar x)-\h V_\g(\bar t,\bar r,\bar x)\ges\eta>0.\ee
Take constants $G>0$ large enough and $\a \in (0,1)$ small enough, so that the following hold:
\bel{u52} \a G < 1, \qq  2\a (\lan\bar{x}\ran +  G C) <\eta /2, \qq \a(G\ell_0-2C_0) >0. \ee
Here $\lan \bar{x} \ran = \sqrt{1+|\bar{x}|^2}$. For any $\a, \b,\e ,\l, \m, \n \in (0,1)$, define
$$ \left\{ \2n\ba{ll}
\ds \f(t,r,x,s,u,y)=\a\(1-{t+s\over4T}\)(\lan x \ran+\lan y \ran)-\b(t+s)\\
\ns\ds\qq\qq\q+{1\over2\l}|x-y|^2+{1\over2\m}|t-s|^2+{1\over2\n}|r-u|^2+{\k\over t}+{\k\over s}+{\th\over r}+{\th\over u}\\
\ns\ds \F(t,r,x,s,u,y)=(1-\a G)V^{\g}(t,r,x)-\h{V}_{\g}(s,u,y)-\varphi(t,r,x,s,u,y),\\
\ns\ds\qq\qq\qq\qq\qq\qq\qq\forall (t,r,x), (s,u,y) \in [0, T] \times [0, T] \times \dbR^n. \ea \right.$$
By \rf{semi1}, we have
$$\left\{\2n\ba{ll}
\ds\lim_{|x|+|y|\to\infty}\F(t,r,x,s,u,y)=-\i,\ \hb{uniformly in~}t,s\in
(0,T],~r,u\in(0,T],\\
\ns\ds\lim_{t\land s\da0}\F(t,r,x,s,u,y)=-\i,\q\hb{uniformly in~}x,y\in
\dbR^n,~r,u\in(0,T],\\
\ns\ds\lim_{r\land u\da0}\F(t,r,x,s,u,y)=-\i,\q\hb{uniformly in~}x,y\in
\dbR^n,~t,s\in(0,T].\ea\right.$$

Thus, there exists a $(t_0,r_0,x_0,s_0,u_0,y_0)\in\{(0,T]\times(0,T]\times\dbR^n\}^2$ (depending on the parameters $\a, \b,\l, \m, \n, \k, \th$ and $\g$) such that
$$\ba{ll}
\ns\ds\F(t_0,r_0,x_0,s_0,u_0,y_0)=\2n\max\limits_{\{(0, T]\1n \times\1n(0, T]\1n \times\1n \dbR^n \}^2}\F(t,r,x,s,u,y)\ges \F(T,T,0,T,T,0)\\
\ns\ds =(1-\a G)V^{\g}(T,T,0)-\h{V}_{\g}(T,T,0)+2\b T -\frac{2}{T}(\k+\th) \ea$$
This, together with  \rf{semi1}, yields the following
\bel{ax}\ba{ll}
\ns\ds \a (\lan x_0\ran\1n+\1n \lan y_0\ran)\1n+\1n\frac{1}{2\l}|x_0\1n-\1n y_0|^2\1n+\1n\frac{1}{2\m}|t_0\1n-\1n s_0|^2\1n +\1n {1\over2\n}|r_0-u_0|^2\\
\ns\ds + \frac{\k(T-t_0)}{t_0 T}+\frac{\k(T-s_0)}{s_0 T}+\frac{\th(T-r_0)}{r_0\d}+\frac{\th(T-u_0)}{u_0 \d}\les M \ea \ee
for some constant $M>0$, independent of $\a, \b,\l, \m, \n, \k, \th$ and $\g$. Consequently, there is a constant $M_{\a}$(independent of $\b,\l, \m, \n,\k, \th$ and $\g$) such that
\bel{u22}\left\{\ba{ll}
\ns\ds|x_0|+ |y_0|+\frac{1}{2\l}|x_0-y_0|^2+\frac{1}{2\m}|t_0-s_0|^2+{1\over2\n}|r_0-u_0|^2 \les M_{\a},\\
\ns\ds \frac{\k T}{M T+\k} \les t_0, s_0 \les T,\qq \frac{\th T}{MT+\th} \les r_0,u_0 \les T. \ea \right. \ee
Next, from the inequality
$$2\F(t_0,r_0,x_0,s_0,u_0,y_0) \ges \F(t_0,r_0,x_0,t_0,r_0,x_0)+\F(s_0,u_0,y_0,s_0,u_0,y_0).$$
along with \rf{u22} and \rf{semi11}, it follows that
\bel{u32} \ba{ll}
\ds \frac{1}{2\l}|x_0- y_0|^2+\frac{1}{2\m}|t_0- s_0|^2+{1\over2\n}|r_0-u_0|^2\\
\ns\ds\les(1-\a G)|V^{\g}(t_0,r_0,x_0)\1n- \1n V^{\g}(s_0,u_0,y_0)|\1n+\1n|\h{V}_{\g}(t_0,r_0,x_0)\1n-\1n \h{V}_{\g}(s_0,u_0,y_0)|\\
\ns\ds \les 2C\{|x_0-y_0|+|t_0-s_0|^{\frac{1}{2}}+|r_0-u_0|^{\frac{1}{2}}\} \to 0, \qq  \text{~as~} \l,\m,\n\to 0. \ea \ee
Note that $(t_0,r_0,x_0,s_0,u_0,y_0)$ depends on the parameters $G, \a, \b,\l, \m, \n,\k, \th$ and $\g$. We claim that for any $(\a,\b,\l, \m, \n,\k, \th, \g)$ small and $G$ large enough, the following cannot be true:
\bel{u42}t_0\lor s_0=T.\ee
In fact, if the above is true, then
\bel{u102} \ba{ll}
\ns\ds (1\1n-\1n\a G)V^{\g}(\bar{t},\bar{r},\bar{x})\1n-\1n\h{V}_{\g}(\bar{t},\bar{r},\bar{x})\1n-\1n2\a \(1\1n-\1n\frac{\bar{t}}{2T}\)\lan \bar{x} \ran\1n+\1n 2\b \bar{t}\1n-\1n \frac{2\k}{\bar{t}}\1n-\1n\frac{2\th}{\bar{r}}\\
\ns\ds=\F(\bar{t},\bar{r},\bar{x},\bar{t},\bar{r},\bar{x})\les\F(t_0,r_0,x_0,s_0,u_0,y_0)\\
\ns\ds\les(1-\a G)V^{\g}(t_0,r_0,x_0)-\hat{V}_{\g}(s_0,r_0,y_0)+2\b T. \ea \ee
Now we send $\l, \m, \n \to 0$. By \rf{u22}, \rf{u32} and \rf{u42}, some subsequence of $(t_0,r_0,x_0,s_0,u_0,y_0)$ converges and the limit has to be of the form $(T,\bar{r}_0, \bar{x}_0, T, \bar{r}_0, \bar{x}_0)$. Then \rf{u102} becomes
$$ \ba{ll}
\ds (1\1n-\1n\a G)V^{\g}(\bar{t},\bar{r},\bar{x})\1n-\1n\h{V}_{\g}(\bar{t},\bar{r},\bar{x})\1n-\1n2\a \(1\1n-\1n\frac{\bar{t}}{2T}\)\lan \bar{x} \ran\1n+\1n 2\b \bar{t}\1n-\1n \frac{2\k}{\bar{t}}\1n-\1n\frac{2\th}{\bar{r}}\\
\ns\ds \les  (1-\a G) V^{\g}(T,\bar{r}_0, \bar{x}_0)-\h{V}_{\g}(T,\bar{r}_0, \bar{x}_0)+2\b T. \ea $$
Next, by sending $\g \to 0$ and using\rf{semi1}, \rf{semi8} and \rf{u52}, we obtain
$$\ba{ll}
\ds\eta \les V^{\g}(\bar{t},\bar{r},\bar{x})-\h{V}_{\g}(\bar{t},\bar{r},\bar{x})\\
\ns\ds  \les 2\a \lan\bar{x}\ran  + 2\b T+2(\frac{\k}{\bar{t}}+\frac{\th}{\bar{r}}) +2\a GC < 2\b T +2(\frac{\k}{\bar{t}}+\frac{\th}{\bar{r}})+ \frac{\eta}{2}. \ea$$
Finally, by sending $ \b, \k, \th\to 0$, we obtain a contradiction. That means for any  $ (\a,\b ,\l, \m, \n,\k, \th, \g)$ small enough and $G$ large enough, we have $t_0,s_0 \in (0, T).$ From the definition of viscosity solution, we have $r_0,u_0 \in (0, T),$ since $r_0\lor u_0=T$ if and only if $t_0\lor s_0=T$.

Next, we claim that
\bel{u62} \hat{V}_\g(s_0,u_0,y_0) < \h N_\g[\hat{V}_\g](s_0,u_0,y_0).\ee
In fact, if
$$\h{V}_\g(s_0,u_0,y_0)=\h N_\g[\h{V}_\g](s_0,u_0,y_0),$$
we send $\g \to 0$ and by \rf{semi3}, \rf{HN}, some subsequence of $( s_0,u_0, y_0)$, still denoted by itself, converges. Then
$$\h{V}(s_0,u_0,y_0)=\h N[\h{V}](s_0,u_0,y_0)=\h{V}(s_0,0,y_0+\x_0)+\ell(s_0,\x_0),$$
for some $\x_0 \in K$. Note that, by \rf{vr7}, we have
\bel{01} V(t_0,\th^{\frac{1}{2}},x_0+\x_0)- V(t_0,r_0,x_0) \ges -\ell(t_0,\x_0)-C\th^{\frac{1}{4}}, \ee
\bel{02} \h{V}(s_0,\th^{\frac{1}{2}},y_0+\x_0)-\h{V}(s_0,u_0,y_0) \les \ell(s_0,\x_0). \ee
Thus, by \rf{01} and \rf{02}, we obtain
\be \ba{ll}
\ds \lim_{\g \to 0} \F(t_0, \th^{\frac{1}{2}}, x_0+\x_0, s_0, \th^{\frac{1}{2}}, y_0+\x_0)-\F(t_0,r_0, x_0, s_0,u_0, y_0)\\
\ns\ds =\1n(1\1n-\1n\a G)[V(t_0,\th^{\frac{1}{2}}\1n,x_0\1n+\1n\x_0)\2n-\2n V(t_0,r_0,x_0)]\2n-\2n[\h{V}(s_0,\th^{\frac{1}{2}}\1n,y_0\1n+\1n\x_0)\2n-\2n\h{V}(s_0,u_0,y_0)]\\
\ns\ds -\a(1\1n-\1n\frac{t_0\1n +\1n s_0}{4T})(\lan x_0\1n +\1n\x_0 \ran\1n-\1n\lan x_0\ran\1n+\1n\lan y_0\1n+\1n\x_0\ran\1n-\1n\lan y_0\ran)\1n-\1n 2\th^{\frac{1}{2}}\2n+\2n \frac{\th}{r_0}\2n+\2n\frac{\th}{u_0}\2n+\2n\frac{1}{2\nu}|r_0\1n-\1n u_0|^2\\
\ns\ds \ges -(1-\a G)(\ell(t_0,\x_0)+C\th^{\frac{1}{4}})+\ell(s_0,\x_0)-2\th^{\frac{1}{2}}\\
\ns\ds -\a\(1-\frac{t_0+ s_0}{4T}\)\(\lan x_0+\x_0\ran-\lan x_0\ran+\lan y_0+\x_0\ran-\lan y_0\ran\). \ea\ee
Now, we send $\l,\m,\nu \to 0$. Some subsequence of $(t_0,r_0,x_0,s_0,u_0,y_0, \x_0)$ converges and the limit has to be of the form $(\bar{t}_0,\bar{r}_0, \bar{x}_0,\bar{t}_0,\bar{u}_0, \bar{x}_0, \bar{\x}_0)$ by \rf{u22} and \rf{u32}. Then,  with \rf{ell}, \rf{K_r^0} and \rf{ax}, we obtain
\be\ba{ll}
\ds \lim_{\l,\m,\n,\th \to 0}\lim_{\g \to 0}\F(t_0,\th^{\frac{1}{2}}, x_0+\x_0, s_0,\th^{\frac{1}{2}}, y_0+\x_0)-\F(t_0,r_0, x_0, s_0,u_0, y_0)\\
\ns \ds \ges \a G \ell(\bar{t}_0,\bar{\x}_0)-2\a|\bar{\x}_0|\ges \a(G\ell_0-2C_0) >0. \ea\ee
This contradicts the definition of $(t_0,r_0, x_0,s_0,u_0, y_0)$. Hence, \rf{u62} holds. Now for fixed
$\a,\k\in (0,1)$, define
$$\ba{ll}
\ds Q \triangleq \Big\{(t,r,x,s,u,y) \in \{[0,T]\times[0,T]\times \dbR^n \}^2 ~|~ t,s\ges\frac{\k T}{2MT+\k},\\
\ns\ds \qq\qq\qq\qq\qq\qq\qq\qq r,u\ges\frac{\th \d}{2M\d+\th},\, |x|,|y|\les 2M_\a\,\Big\}, \ea$$
with $M_{\a}$ being the same as that appearing in \rf{u22}. Clearly, $\f(t,r,x,s,u,y)$ is semiconcave on $Q$ and therefore, $\F(t,r,x,s,u,y)$ is semiconvex with maximum value at $(t_0,r_0,x_0,s_0,u_0,y_0)$ in the interior of $Q$(noting \rf{u22}). Hence for any small $\om>0$,
$$\ba{ll}
\ds \h\F(t,r,x,s,u,y)\deq\F(t,r,x,s,u,y)-\om(|t-t_0|^2+|s-s_0|^2+|r-r_0|^2\\
\ns\ds\qq\qq\qq\qq\qq\qq\qq\qq+|u-u_0|^2+|x-x_0|^2
+|y-y_0|^2)\ea$$
is semiconvex on $Q$, attaining a strict maximum at $(t_0,r_0,x_0,s_0,u_0,y_0).$ By Alexandrov's theorem and Jensen's lemma, for the above given $\om>0$, there exist $q,\hat{q}.l,\hat{l} \in \dbR$ and $p,\hat{p} \in \dbR^n$ with
\bel{U52}
|q|+|\hat{q}|+|l|+|\hat{l}|+|p|+|\hat{p}|\les \om,\ee
and $(\hat{t}_0,\hat{r}_0,\hat{x}_0,\hat{s}_0,\hat{u}_0,\hat{y}_0) \in Q$ with
\bel{U62}
|\hat{t}_0-t_0|+|\hat{r}_0-r_0|+|\hat{x}_0-x_0|+|\hat{s}_0-s_0|+|\hat{u}_0-u_0|+|\hat{y}_0-y_0|\les \om
\ee
such that
$$\ba{ll}
\ns\ds \h{\F}(t,x,s,y)+qt+\hat{q}s+lr+\hat{l}u+\lan p, x \ran +\lan \hat{p}, y \ran  \\
\ns\ds \equiv (1\1n-\1n\a G)V^{\g}(t,r,x)\1n-\1n\h{V}_{\g}(s,u,y)\1n-\1n\f(t,r,x,s,u,y)\1n-\1n\om(|t\1n-\1n t_0|^2\1n+\1n|s\1n-\1n s_0|^2\1n \\
\ns\ds\q +\1n|r\1n-\1n r_0|^2\1n+\1n|u\1n-\1n u_0|^2\1n+\1n|x\1n-\1n x_0|^2\1n+\1n|y\1n-\1n y_0|^2)\1n+\1n qt\1n+\1n\hat{q}s\1n+\1n l r\1n+\1n\hat{l}u\1n+\1n\lan p, x \ran\1n +\1n\lan \hat{p}, y \ran  \ea$$
attains a maximum at $(\hat{t}_0,\hat{r}_0,\hat{x}_0,\hat{s}_0,\hat{u}_0,\hat{y}_0)$, at which $(1-\a G)V^{\g}(t,r,x)-\h{V}_{\g}(s,u,y)$ is twice differentiable. For notational simplicity, we now drop $\g$ in $V^{\g}(t,r,x)$ and $\hat{V}_{\g}(s,u,y)$. Then, by the first- and second-order necessary conditions for a maximum point, at the point $(\hat{t}_0,\hat{r}_0,\hat{x}_0,\hat{s}_0,\hat{u}_0,\hat{y}_0)$, we must have
\bel{U32}
\left\{\2n \ba{ll}
\ns\ds V_t\1n=\1n\frac{1}{1\1n-\1n\a G}[\f_t\1n+\1n2\om(\hat{t}_0\1n-\1n t_0)\1n-\1n q],\q\h{V}_s\1n=-\1n\f_s\1n-\1n2\om(\hat{s}_0\1n-\1n s_0)\1n+\1n\hat{q},\\
\ns\ds V_r\1n=\1n\frac{1}{1\1n-\1n\a G}[\f_r\1n+\1n2\om(\hat{r}_0\1n-\1n r_0)\1n-\1nl],\q\h{V}_{u}\1n=-\1n\f_{u}\1n-\1n2\om(\hat{u}_0\1n-\1n u_0)\1n+\1n\hat{l},\\
\ns\ds V_x\1n=\1n\frac{1}{1\1n-\1n\a G}[\f_x\1n+\1n2\om(\hat{x}_0\1n-\1n x_0)\1n-\1n p],\q\h{V}_y\1n=-\f_y-2\om(\hat{y}_0\1n-\1n y_0)\1n+\1n\hat{p},\\
\ns\ds \begin{bmatrix}(1-\a G)V_{xx} & 0 \\ 0 & -\h{V}_{yy} \end{bmatrix} \les \begin{bmatrix}\f_{xx}+2\om I_n & \f_{xy}\\ \f^\top_{xy} & \f_{yy}+ 2\om I_n \end{bmatrix} , \ea \right. \ee
where $I_{2n}$ is the $2n \times 2n$ identity matrix. Now, at point  $(\hat{t}_0,\hat{r}_0,\hat{x}_0,\hat{s}_0,\hat{r}_0,\hat{y}_0)$, we calculate the following:
\bel{U42}
\left\{ \ba{ll}
\ns\ds \f_t= -\b-\frac{\k}{(\hat{t}_0)^2}-\frac{\a}{4T}(\lan\hat{x}_0\ran+\lan\hat{y}_0\ran)+\frac{1}{\m}(\hat{t}_0-\hat{s}_0),\\
\ns\ds \f_s= -\b-\frac{\k}{(\hat{s}_0)^2}-\frac{\a}{4T}(\lan\hat{x}_0\ran+\lan\hat{y}_0\ran)+\frac{1}{\m}(\hat{s}_0-\hat{t}_0),\\
\ns\ds \f_r=-\frac{\th}{(\hat{r}_0)^2}+\frac{1}{\n}(\hat{r}_0-\hat{u}_0),\\
\ns\ds \f_u=-\frac{\th}{(\hat{u}_0)^2}+\frac{1}{\n}(\hat{u}_0-\hat{r}_0),\\
\ns\ds \f_x= \a (1-\frac{\hat{t}_0+\hat{s}_0}{4T})\frac{\hat{x}_0}{\lan\hat{x}_0\ran}+\frac{\hat{x}_0-\hat{y}_0}{\l},\\
\ns\ds \f_y= \a (1-\frac{\hat{t}_0+\hat{s}_0}{4T})\frac{\hat{y}_0}{\lan\hat{y}_0\ran}+\frac{\hat{y}_0-\hat{x}_0}{\l},\\
\ns\ds A\equiv\begin{bmatrix} \varphi_{xx} & \varphi_{xy} \\ \varphi_{xy}^\top & \varphi_{yy} \end{bmatrix}=\frac{1}{\l}\begin{bmatrix} I_n &-I_n \\ -I_n & I_n \end{bmatrix}\\
\ns\ds\qq\qq\qq\qq+\1n\a (1\1n-\1n\frac{\hat{t}_0\1n+\1n\hat{s}_0}{4T}) \begin{bmatrix} \frac{I}{\lan\hat{x}_0\ran}\1n-\1n\frac{xx^\top}{\lan\hat{x}_0\ran^3} & 0 \\ 0 & \frac{I}{\lan\hat{y}_0\ran}\1n-\1n\frac{yy^\top}{\lan\hat{y}_0\ran^3} \end{bmatrix}. \ea \right. \ee
On the other hand, by Lemma A.3, \rf{u62} and the definition of viscosity sub- and super-solutions, we have
$$\left\{\3n \ba{ll}
\ns\ds V_t(\hat{t}_0,\hat{r}_0,\hat{x}_0)\1n+\1n V_r(\hat{t}_0,\hat{r}_0,\hat{x}_0)\1n+\1n H^{\g}(\hat{t}_0,\hat{x}_0,V_x(\hat{t}_0,\hat{r}_0,\hat{x}_0),V_{xx}(\hat{t}_0,\hat{r}_0,\hat{x}_0))\1n\ges\1n 0, \\
\ns\ds \h{V}_s(\hat{s}_0,\hat{u}_0,\hat{y}_0)\1n+\1n\h{V}_u(\hat{s}_0,\hat{u}_0,\hat{y}_0)\1n+\1n H_{\g}(\hat{s}_0,\hat{y}_0,\h{V}_y(\hat{s}_0,\hat{u}_0,\hat{y}_0),\h{V}_{yy}(\hat{s}_0,\hat{u}_0,\hat{y}_0))\1n\les\1n 0, \ea \right.$$
By \rf{U2}, one can find a $(\bar{t}_0,\bar{x}_0,\bar{s}_0,\bar{y}_0)$ with
\bel{U112}
|\bar{t}_0-\hat{t}_0|+|\bar{x}_0-\hat{x}_0|+|\bar{s}_0-\hat{s}_0| +|\bar{y}_0-\hat{y}_0| \les C\g, \ee
for some $C>0$, such that
\bel{U72}\3n\3n \ba{ll}
\ns\ds \h{V}_s(\hat{s}_0,\hat{u}_0,\hat{y}_0)\1n+\1n\h{V}_{u}(\hat{s}_0,\hat{u}_0,\hat{y}_0)\1n-\1n(1\1n-\1n\a G)\(V_t(\hat{t}_0,\hat{r}_0,\hat{x}_0)\1n+\1n V_r(\hat{t}_0,\hat{r}_0,\hat{x}_0)\)\\
\ns\ds\les(1-\a G)H^{\g}(\hat{t}_0,\hat{x}_0,V_x(\hat{t}_0,\hat{r}_0,\hat{x}_0), V_{xx}(\hat{t}_0,\hat{r}_0,\hat{x}_0))\\
\ns\ds \qq\qq -H_{\g}(\hat{s}_0,\hat{y}_0,\h{V}_y(\hat{s}_0,\hat{u}_0,\hat{y}_0),\h{V}_{yy}(\hat{s}_0,\hat{u}_0,\hat{y}_0))\\
\ns\ds = (1-\a G)H(\bar{t}_0,\bar{x}_0,V_x(\hat{t}_0,\hat{r}_0,\hat{x}_0), V_{xx}(\hat{t}_0,\hat{r}_0,\hat{x}_0))\\
\ns\ds\qq\qq-H(\bar{s}_0,\bar{y}_0,\h{V}_y(\hat{s}_0,\hat{u}_0,\hat{y}_0),\h{V}_{yy}(\hat{s}_0,\hat{u}_0,\hat{y}_0)) \\
\ns\ds =\frac{1}{2}\tr\[\si(\bar{t}_0,\bar{x}_0)^\top (1-\a G)V_{xx}(\hat{t}_0,\hat{r}_0,\hat{x}_0)\si(\bar{t}_0,\bar{x}_0)\\
\ns\ds \qq\qq\qq\qq-\si(\bar{s}_0,\bar{y}_0)^\top\h{V}_{yy}(\hat{s}_0,\hat{u}_0,\hat{y}_0)\si(\bar{s}_0,\bar{y}_0)\]\\
\ns\ds \qq+\[\lan(1\1n-\1n\a G)V_x(\hat{t}_0,\hat{r}_0,\hat{x}_0), b(\bar{t}_0,\bar{x}_0)  \ran\1n-\1n\lan \h{V}_y(\hat{s}_0,\hat{u}_0,\hat{y}_0), b(\bar{s}_0,\bar{y}_0) \ran\]\\
\ns\ds\qq+\[(1-\a G)g(\bar{t}_0,\bar{x}_0)- g(\bar{s}_0,\bar{y}_0)\]\\
\ns\ds \equiv (I)+(II)+(III). \ea \ee
By \rf{U52}-\rf{U42}, we have
$$ \ba{ll}
\ns \ds \h{V}_s(\hat{s}_0,\hat{u}_0,\hat{y}_0)\1n+\1n\h{V}_{u}(\hat{s}_0,\hat{u}_0,\hat{y}_0)\1n-\1n(1\1n-\1n\a G)V_t(\hat{t}_0,\hat{r}_0,\hat{x}_0)\1n-\1n(1\1n-\1n\a G)V_r(\hat{t}_0,\hat{r}_0,\hat{x}_0)\\
\ns \ds =2\b+\frac{\a}{2T}(|\hat{x}_0|^2+|\hat{y}_0|^2)+\frac{\k}{(\hat{t}_0)^2}+\frac{\k}{(\hat{s}_0)^2}+\frac{\th}{(\hat{r}_0)^2}+\frac{\th}{(\hat{u}_0)^2}\\
\ns\ds \qq-2\om(\hat{t}_0-t_0+\hat{r}_0-r_0+\hat{s}_0-s_0+\hat{u}_0-u_0)+\hat{q}+\hat{l}+q+l\\
\ns\ds \ges 2\b+\frac{\a}{2T}(|\hat{x}_0|^2+|\hat{y}_0|^2)-M\om, \ea$$
for some absolute constant $M>0$. By \rf{u32} and \rf{U62}, we see that one may assume that as $\l,\m,\n,\om \to 0$, $(\hat{t}_0,\hat{r}_0,\hat{x}_0)$ and $(\hat{s}_0,\hat{u}_0,\hat{y}_0)$ converge to the same limit, denoted by $(t_{\a},r_\a,x_{\a})$, to emphasize the dependence on $\a$. Thus, letting $\l,\m,\n,\om \to 0$ in the above leads to
\bel{U82}\3n \ba{ll}
\ns\ds \h{V}_s(t_{\a},\1n r_\a,\1n x_{\a})\1n+\1n\h{V}_{u}(t_{\a},\1n r_\a,\1n x_{\a})\1n-\1n(1\1n-\1n\a G)\(\1nV_t(t_{\a},\1n r_\a,\1n x_{\a})\1n+\1n V_r(t_{\a},\1n r_\a,\1n x_{\a})\1n\)\\
\ns\ds \ges 2\b+\frac{\a}{T}|\hat{x}_{\a}|^2, \ea\ee
This gives an estimate for the left-hand side of \rf{U72}.  Now we estimate the terms (I),(II),(III) on the right side of \rf{U72} one by one. First of all, from  \rf{u32},\rf{U62},\rf{U112} and the continuity of $g(t,x)$, one may obtain an estimate for (III):
\bel{U122}
\lim_{\substack{\l,\m, \g,\om \to 0}}(III) \triangleq (1- \a G)g(\bar{t}_0,\bar{x}_0)-g(\bar{s}_0,\bar{y}_0)\les \a GL.\ee
Next,
$$ \ba{ll}
\ns\ds (II)  \triangleq  \lan (1-\a G)V_x(\hat{t}_0,\hat{r}_0,\hat{x}_0), b(\bar{t}_0,\bar{x}_0)  \ran -\lan \h{V}_y(\hat{s}_0,\hat{u}_0,\hat{y}_0), b(\bar{s}_0,\bar{y}_0) \ran\\
\ns\ds  =\lan \a (1-\frac{\hat{t}_0+\hat{s}_0}{4T})\frac{\hat{x}_0}{\lan\hat{x}_0\ran}+\frac{\hat{x}_0-\hat{y}_0}{\l}+2\om(\hat{x}_0-x_0)-p, b(\bar{t}_0,\bar{x}_0)  \ran\\
\ns \ds  \q +\lan \a (1-\frac{\hat{t}_0+\hat{s}_0}{4T}))\frac{\hat{y}_0}{\lan\hat{y}_0\ran}+\frac{\hat{y}_0-\hat{x}_0}{\l}+2\om(\hat{y}_0-y_0)-\hat{p}, b(\bar{s}_0,\bar{y}_0) \ran\\
\ns \ds\les 2\a L (1\1n-\1n\frac{\hat{t}_0\1n+\1n\hat{s}_0}{4T})(|\h{x}_0|\1n+\1n|\h{y}_0|)\1n+\1n \om(1\1n+\1n 2\om)L\1n +\1n\lan \frac{\hat{x}_0\1n-\1n\hat{y}_0}{\l}, b(\bar{t}_0,\bar{x}_0)\1n-\1n b(\bar{s}_0,\bar{y}_0) \ran .\ea $$
Letting $\m,\om,\g \to 0$, from \rf{u32}, \rf{U62} and \rf{U112}, we may assume that $(\hat{t}_0,\hat{x}_0,\hat{s}_0,\hat{y}_0)$ and $(\bar{t}_0,\bar{x}_0,\bar{s}_0,\bar{y}_0)$ converge to the same limit, which is denoted by
$(t_{0},x_{0},t_{0},y_{0})$. Thus
$$\lim_{\substack{\m,\om,\g \to 0}}(II) \les  2\a L (1-\frac{t_0}{2T})(|x_0|+|y_0|)+L\frac{|x_0-y_0|^2}{\l}$$
Then let $\l \to 0$, one concludes that $(t_0,x_0)$ and $(t_0,y_0)$ approach a common limit, called $(t_{\a},x_{\a})$. Consequently,
\bel{U92}
\lim_{\substack{\l \to 0}}\lim_{\substack{\m,\om,\g \to 0}}(II) \les 4\a L(1-\frac{t_{\a}}{2T})|x_{\a}|. \ee
Now we treat (I) in \rf{U72}. By the inequality in \rf{U32},
$$\ba{ll}
\ns\ds (I)\deq{1\over2}\tr\[\si(\bar{t}_0,\bar{x}_0)^\top(1-\a G)V_{xx}(\hat{t}_0,\hat{r}_0,\hat{x}_0)\si(\bar{t}_0,\bar{x}_0)\\
\ns\ds \qq\qq\qq\qq\qq-\si(\bar{s}_0,\bar{y}_0)^\top \h{V}_{yy}(\hat{s}_0,\hat{u}_0,\hat{y}_0)\si(\bar{s}_0,\bar{y}_0)\]\\
\ns\ds \1n ={1\over2}\tr \1n\[ \1n\begin{bmatrix} \si(\bar{t}_0,\bar{x}_0) \\\si(\bar{s}_0,\bar{y}_0) \end{bmatrix}^\top \2n\begin{bmatrix} (1\1n-\1n\a G)V_{xx}(\hat{t}_0,\hat{r}_0,,\hat{x}_0) & 0 \\ 0 & -\1n\h{V}_{yy}(\hat{s}_0,\hat{u}_0,\hat{y}_0) \end{bmatrix}\1n \begin{bmatrix} \si(\bar{t}_0,\bar{x}_0) \\ \si(\bar{s}_0,\bar{y}_0) \end{bmatrix}\]\\
\ns\ds \les {1\over2}\tr \[ \begin{bmatrix} \si(\bar{t}_0,\bar{x}_0) \\\si(\bar{s}_0,\bar{y}_0) \end{bmatrix}^\top [A+2 \om I_{2n}] \begin{bmatrix} \si(\bar{t}_0,\bar{x}_0) \\\si(\bar{s}_0,\bar{y}_0) \end{bmatrix}\]\\
\ns\ds\les \1n{1\over2}\{\1n\frac{1}{\l}|\si(\bar{t}_0,\1n\bar{x}_0)\1n-\1n\si(\bar{s}_0,\1n\bar{y}_0)|^2\2n +\1n[\a (1\1n-\1n\frac{\hat{t}_0\1n+\1n\hat{s}_0}{4T})\1n+\1n 2\om](|\si(\bar{t}_0,\1n\bar{x}_0)|^2\2n+\1n|\si(\bar{s}_0,\1n\bar{y}_0)|^2)\}\\
\ns \ds \les {1\over2\l}|\si(\bar{t}_0,\bar{x}_0)-\si(\bar{s}_0,\bar{y}_0)|^2+2[\a (1-\frac{\hat{t}_0+\hat{s}_0}{4T})+2\om]L^2. \ea$$
As above, we first let  $\m,\om,\g \to 0$ and then let $\l \to 0$ to get
\bel{U102}
\lim_{\substack{\l \to 0}}\lim_{\substack{\m,\om,\g \to 0}}(I) \les 2\a L^2(1-\frac{t_{\a}}{2T}). \ee
Combining \rf{U72} - \rf{U102}, we obtain
$$2\b+\frac{\a}{T}|x_{\a}|^2\les \a GL + 4\a L(1-\frac{t_{\a}}{2T})|x_{\a}|+2\a L^2(2-\frac{t_{\a}}{2T}).$$
Note that $t_{\a} \in (0,T)$, we have
\bel{u112}\b\les\a\Big\{-\frac{1}{2T}|x_{\a}|^2+2L|x_{\a}|+2L^2+\frac{GL}{2}\Big\}.\ee
It is clear that the term inside the braces on the right-hand side of \rf{u112} is bounded from above uniformly in $\a$. Thus, by sending $\a \to 0$, we obtain $\b \les 0$, which contradicts our assumptions $\b >0$. This proves \rf{contra}.
\end{proof}

\end{document}